\documentclass{amsart}
\pdfoutput=1 

\usepackage{amsmath,amsthm,amssymb}
\usepackage{lipsum}
\usepackage{amsfonts}
\usepackage{graphicx}
\usepackage{epstopdf}
\usepackage{algorithmic}
\ifpdf
  \DeclareGraphicsExtensions{.eps,.pdf,.png,.jpg}
\else
  \DeclareGraphicsExtensions{.eps}
\fi

\usepackage[a4paper,margin=3cm]{geometry}%
\usepackage{lmodern,dsfont,soulutf8,mathrsfs}%
\usepackage[showonlyrefs=true]{mathtools}%
\usepackage[utf8]{inputenc}%
\usepackage[T1]{fontenc}%
\usepackage{listings}%
\usepackage{hyperref}

\newtheorem{theorem}{Theorem}[section]%
\newtheorem{corollary}[theorem]{Corollary}%
\newtheorem{proposition}[theorem]{Proposition}%
\newtheorem{lemma}[theorem]{Lemma}%
\newtheorem{remark}[theorem]{Remark}%
\newtheorem{assumption}[theorem]{Assumption}%
\newtheorem{algo}[theorem]{Algo}%

\title[Coulomb gases under constraints -- Chafaï, Ferré, Stoltz]%
{Coulomb gases under constraint:\\some theoretical and numerical results}

\thanks{Accepted in SIAM Journal on Mathematical Analysis (SIMA), September 30, 2020. The PhD of Grégoire
      Ferré is supported by the Labex Bézout ANR-10-LABX-58-01. The work of
      Gabriel Stoltz was funded in part by the Agence Nationale de la
      Recherche, under grant ANR-14-CE23-0012 (COSMOS). Gabriel Stoltz is
      supported by the European Research Council under the European Union’s
      Seventh Framework Programme (FP/2007-2013)/ERC Grant Agreement number
      614492.}

\author{Djalil Chafaï}
\address{Universit\'e Paris-Dauphine, PSL, CNRS, CEREMADE, F-75016 Paris, France}
\email{djalil(at)chafai.net}
\urladdr{http://djalil.chafai.net/}
\author{Grégoire Ferré}
\address{CERMICS, École des Ponts, Marne-la-Vallée, France and Matherials, Inria Paris, France}
\email{gregoire.ferre(at)ponts.org}
\urladdr{https://sites.google.com/view/gferre/}
\author{Gabriel Stoltz}
\address{CERMICS, Ecole des Ponts, Marne-la-Vallée, France and Matherials, Inria Paris, France}
\email{gabriel.stoltz(at)enpc.fr}
\urladdr{https://cermics.enpc.fr/\~stoltz}

\usepackage{amsopn}

\newcommand{\dP}{\mathbb{P}}
\newcommand{\dQ}{\mathbb{Q}}\newcommand{\dR}{\mathbb{R}}

\newcommand{\cE}{\mathcal{E}}

\newcommand{\cM}{\mathcal{M}}
\newcommand{\cP}{\mathcal{P}}

\newcommand{\Pc}{\mathcal{P}_{\mathrm{c}}}

\newcommand{\bE}{\mathbf{E}}

\newcommand{\Dt}{\Delta t}
\newcommand{\Nit}{N_{\mathrm{iter}}}

\newcommand{\CEIL}[1]{{{\lceil#1\rceil}}} %

\newcommand{\mus}{\mu_{\star}}
\newcommand{\Km}{K_{\mathrm{max}}}

\newcommand{\ind}{\mathds{1}}

\newcommand{\DBL}{\mathrm{d}_{\mathrm{BL}}}
\newcommand{\DW}{\mathrm{d}_{\mathrm{W}_1}}
\newcommand{\DWP}{\mathrm{d}_{\mathrm{W}_p}}



\begin{document}

\begin{abstract}
  We consider Coulomb gas models for which the empirical measure typically
  concentrates, when the number of particles becomes large, on an equilibrium
  measure minimizing an electrostatic energy. We study the behaviour when the
  gas is conditioned on a rare event. We first show that the special case of
  quadratic confinement and linear constraint is exactly solvable due to a
  remarkable factorization, and that the conditioning has then the simple
  effect of shifting the cloud of particles without deformation. To address
  more general cases, we perform a theoretical asymptotic analysis relying on
  a large deviations technique known as the Gibbs conditioning principle. The
  technical part amounts to establishing that the conditioning ensemble is an
  $I$-continuity set of the energy. This leads to characterizing the
  conditioned equilibrium measure as the solution of a modified variational
  problem. For simplicity, we focus on linear statistics and on quadratic
  statistics constraints. Finally, we numerically illustrate our predictions
  and explore cases in which no explicit solution is known. For this, we use a
  Generalized Hybrid Monte Carlo algorithm for sampling from the conditioned
  distribution for a finite but large system.
\end{abstract}

\maketitle

{\footnotesize\tableofcontents}

\section{Introduction}
\label{sec:intro}
This section contains the main elements of the considered model, some
motivations and the plan of the paper. We consider here the so-called Coulomb
gas model that, in addition to its physical interest, shows an interesting
behaviour in the limit of a large number of particles, see for
instance~\cite{MR3262506,serfaty-icm2018,leble2017large}. The model consists
of a set of random particles $X_{n,1},\ldots,X_{n,n}$ for $n\geq 2$, where
each~$X_{n,i}$ belongs to~$\dR^d$ for some physical dimension $d\geq 2$. The
particles interact through the Coulomb kernel $g:\dR^d\to \dR$ defined by
\[
  g(x) = \left\{
    \begin{aligned}
      &  \log \frac{1}{|x|}, & \mathrm{if}\quad d=2,
      \\
      &   \frac{1}{(d-2)|x|^{d-2}}, & \mathrm{if}\quad d\geq 3.
    \end{aligned}
  \right.
\]
This denomination comes from the equation satisfied by the interaction~$g$.
Indeed, denoting by~$\delta_0$ the Dirac mass at~$0$, $g$ solves in the sense
of distributions the following Poisson problem:
\begin{equation}
  \label{eq:poisson}
  -\Delta g = c_d\delta_0, %
  \quad \mathrm{with}\quad %
  c_d = \mathrm{surface}\big(\{x\in\mathbb{R}^d:|x|=1\}\big) %
  = 2\frac{\pi^{d/2}}{\Gamma(d/2)}.
\end{equation}
Note that $\lim_{|x|\to+\infty}g(x)=0$ if $d\geq3$, while
$\lim_{|x|\to+\infty}g(x)=-\infty$ if $d=2$. In \eqref{eq:poisson}
$\Delta=\sum_{i=1}^n\partial^2_i$ denotes the Laplacian operator in~$\dR^d$.
In addition to this pair interaction, the particles are subject to a confining
potential $V:\dR^d\to\dR$ assumed to be lower semi-continuous and such that
\begin{equation}
  \label{eq:confinement}
  \varliminf_{|x|\to+\infty} \big(V(x)-2\ind_{d=2}\log|x|\big)>-\infty.
\end{equation}
Following~\cite{MR3820329} or~\cite{MR3309890}, under this assumption, we can
define the electrostatic energy on~$\mathcal{P}(\mathbb{R}^d)$ by
\begin{equation}\label{eq:cE}
  \mathcal{E}(\mu)
  =\iint_{\dR^d\times \dR^d}\left(g(x-y)+\frac{V(x)+V(y)}{2}\right)\mu(\mathrm{d}x)\mu(\mathrm{d}y).
\end{equation}
This makes sense in $\mathbb{R}\cup\{+\infty\}$ since the integrand is bounded
from below thanks to the assumption~\eqref{eq:confinement} on~$V$. Moreover
for all $\mu\in\mathcal{P}(\mathbb{R}^d)$ such that
\begin{equation}
  \label{eq:log_int_condition}
  \int\log(1+|x|)\mathbf{1}_{d=2} \, \mu(\mathrm{d}x)<\infty,
\end{equation}
we have
\begin{equation}\label{eq:cEbis}
  \mathcal{E}(\mu)=\iint_{\dR^d\times \dR^d} g(x-y)\mu(\mathrm{d}x)\mu(\mathrm{d}y)
  +\int_{\dR^d} V(x)\mu(\mathrm{d}x);
\end{equation}
see~\eqref{eq:def_J} below. The functional~$\mathcal{E}$ has a unique minimizer
on~$\mathcal{P}(\mathbb{R}^d)$ called the \emph{equilibrium
  measure}~\cite{MR3820329,MR3309890}:
\begin{equation}\label{eq:mus}
  \mus=\underset{\mathcal{P}(\mathbb{R}^d)}{\mathrm{argmin}}\ \mathcal{E}.
\end{equation}
It has compact support, and if moreover~$V$ has a Lipschitz continuous
derivative then it has density
\begin{equation}\label{eq:muV}
  \frac{\Delta V}{2c_d}
\end{equation}
on the interior of its support. In particular if~$V$ is proportional
to~$\left|\cdot\right|^2$ then~$\mus$ is uniform on a ball. The compactness of
the support of~$\mus$ comes from the \emph{strong confinement}
assumption~\eqref{eq:confinement}. Note that it is possible to consider
\emph{weakly confining potentials} for which the equilibrium measure still
exists but is no longer compactly supported, see for instance the spherical
ensemble in~\cite{hardy2012note,MR3215627}.

Let $X_n=(X_{n,1},\ldots,X_{n,n})$ be a random vector of~$(\mathbb{R}^d)^n$
with law
\begin{equation}\label{eq:Pn}
  P_n(\mathrm{d} x)=
  \frac{\mathrm{e}^{-\beta_n H_n(x_1,\ldots,x_n)}}{Z_n}\mathrm{d} x_1\cdots\mathrm{d} x_n,
\end{equation}
where $\beta_n>0$ satisfies
\begin{equation}
  \label{eq:betandiverge}
  \lim_{n\to\infty}\frac{\beta_n}{n}=+\infty,
\end{equation}
and
\begin{equation}\label{eq:Hn}
  H_n(x_1,\ldots,x_n) = \frac{1}{n}\sum_{i=1}^nV(x_i)+\frac{1}{n^2}\sum_{i\neq
  j}g(x_i-x_j).
\end{equation}
This makes sense only if 
\begin{equation}\label{eq:Zn}
  Z_n = \int_{(\mathbb{R}^d)^n}\mathrm{e}^{-\beta_n H_n(x_1,\ldots,x_n)}
  \mathrm{d} x_1\cdots\mathrm{d} x_n<\infty,
\end{equation}
which is the case when~$V$ satisfies
\begin{equation}\label{eq:Vcond}
  \int_{\mathbb{R}^d}\mathrm{e}^{-\frac{\beta_n}{n}(V(x)-2\ind_{d=2}\log(1+|x|))}\mathrm{d}x<\infty.
\end{equation}

\begin{remark}
The condition~\eqref{eq:betandiverge} ensures a large deviation principle with
a simple rate function. It is possible to consider the Sanov regime
$\beta_n=\beta n$ for which the rate function has an additional classical
entropy term, see for instance~\cite{MR3262506,david-ihp,MR3825945,dupuis,lambert2019poisson}.
\end{remark}

This model is standard in mathematical physics:~$P_n$ is a Boltzmann--Gibbs
measure modelling a gas of particles, called here a \emph{Coulomb gas}, at
inverse temperature~$\beta_n$ and with Hamiltonian~$H_n$. The law~$P_n$ is
exchangeable in the sense that~$H_n$ is symmetric in $x_1,\ldots,x_n$. Indeed,
it depends on $x_1,\ldots,x_n$ only via the empirical measure, namely,~$P_n$
almost surely,
\begin{equation}
  \label{eq:Hnmun}
  H_n=\int_{\dR^d} V(x)\mu_n(\mathrm{d}x) %
  +\iint_{\neq}g(x-y)\mu_n(\mathrm{d}x)\mu_n(\mathrm{d}y)  
  \quad\text{with}\quad
  \mu_n =\frac{1}{n}\sum_{i=1}^n\delta_{x_i},
\end{equation}
where the double integration runs
over~$\{(x,y)\in\dR^d\times\dR^d\,|\, x\neq y\}$. A heuristic reasoning
suggests that, if $\beta_n\to+\infty$ fast enough, under~$P_n$ the empirical
measure~$\mu_n$ should concentrate in the limit $n\to+\infty$ on the
equilibrium measure~$\mus$ that minimizes the energy~$\mathcal{E}$
in~\eqref{eq:cEbis}-\eqref{eq:mus}. This is intuited from the Laplace
principle given the expression~\eqref{eq:Pn} for~$P_n$, where~$H_n$ is defined
by~\eqref{eq:Hnmun}. This intuition can be made rigorous through a large
deviations principle (LDP), which can be established in this case and many
others, see for instance~\cite{arous1997large,MR3262506,dupuis,david-ihp} and
the references therein. In particular, the case~$d=2$ with quadratic
confinement~$V$ corresponds to the well-known Ginibre ensemble for random
matrices~\cite{MR2641363,MR3699468}. We could also consider more general
interactions, such as Riesz kernels~\cite{MR3262506,leble2017large},
discontinuous~\cite{dupuis} or weak~\cite{hardy2012note} confinement, but we
stick to this setting for ease of presentation. The technical requirements
needed for extending our proofs will be pointed out throughout the paper, and
cases not covered by the theoretical analysis will be investigated numerically
in Section~\ref{sec:numerics}.

As large deviations are concerned with probabilities of rare fluctuations, it
is possible to consider the empirical measure of the random gas conditioned on
such a fluctuation. There has been a number of works on the behaviour of such
gases conditioned on having an unusual proportion of the particles lying in
some region of the space. As an example, for~$d=2$ and~$V$
quadratic,~\cite{MR3264829} reformulates the conditioned equilibrium measure
through an obstacle problem. On the other hand~\cite{ghosh-nishry,MR3825948}
consider the rare situation in which there is a ``hole'' in the distribution,
in other words no particle around zero.
Finally~\cite{PhysRevLett.103.220603,PhysRevE.83.041105} consider the one
dimensional Wigner situation in which an abnormal proportion of particles lie
on one side of the real line. Explicit expressions can be obtained in the
latter case. The study of such conditionings is motivated by questions arising
in theoretical physics, see for instance the references
in~\cite{PhysRevE.83.041105}.

While the above mentioned works bring substantial contributions to the
understanding of conditioned random gas distributions, they also motivate
further questions. Indeed, one may consider more general constraints, like
conditioning on the barycenter of the cloud being far away from the origin.
This may be of interest for both theoretical~\cite{MR3264829} and practical
purposes (if one wants to filter out noise conditioned on some rare
event~\cite{bouchaud2009financial}). Moreover, the numerical methods proposed
in~\cite{ghosh-nishry,MR3825948,PhysRevE.83.041105} do not seem adapted to
sampling the empirical distribution conditioned on some event -- since this
event is typically rare, direct rejection sampling is generally inefficient.
The goal of this paper is therefore to investigate some theoretical results on
such conditioned Coulomb gases, as well as providing an algorithm to
sample conditioned distributions.

Mathematically, our aim is to consider the particles
$Y_n=(Y_{n,1},\ldots,Y_{n,n})$ in~$(\mathbb{R}^d)^n$ such that
\begin{equation}\label{eq:lawYnineq}
  Y_n\sim\mathrm{Law}\big( X_n \,\big|\, \xi_n(X_n) \leq 0 \big),
\end{equation}
where $\xi_n:(\dR^d)^n\to \dR$, and to consider the limiting behaviour of the
empirical measure
\[
\frac{1}{n}\sum_{i=1}^n \delta_{Y_{n,i}},
\]
as $n\to +\infty$, depending on the confinement
potential~$V$ and the constraint~$\xi_n$. Instead of an inequality constraint
like~\eqref{eq:lawYnineq}, we may instead consider an equality constraint
\[
  Y_n\sim\mathrm{Law}\big( X_n \,\big|\, \xi_n(X_n) = 0 \big).
\]
We will generally consider inequality constraints since they naturally lead to
a Gibbs conditioning principle. Equality constraints could be considered as
well by an additional limiting procedure, see~\cite[Section~7.3]{MR2571413}
and the discussion in Section~\ref{sec:LDP}. We could also consider~$\xi_n$ to
be $\dR^m$-valued for some $m\geq 2$ but we restrict to one dimensional
constraints for ease of exposition. The cases studied
in~\cite{MR3264829,ghosh-nishry,MR3825948,PhysRevLett.103.220603,PhysRevE.83.041105}
correspond to the choice
\[
\xi_n(x_1,\hdots, x_n) = \mu_n(\ind_U) - c,
\]
for some measurable set $U\subset \dR^d$ and constant $c\in\dR$. We will study
in this paper more general \emph{linear statistics} of the form
\begin{equation}
  \label{eq:xinlinintro}
\xi_n(x_1,\hdots, x_n) = \mu_n(\varphi),
\end{equation}
for some constraint function $\varphi:\dR^d \to \dR$ satisfying growth
conditions, see Section~\ref{sec:linear}. A particular case of interest is
when the constraint function~$\varphi$ is itself linear, namely:
\begin{equation}\label{eq:linphi}
  \varphi(x) = x\cdot v - c,
\end{equation}
for $v\in\dR^d$ and $c\in\dR$. Indeed, when~$\varphi$ is chosen according
to~\eqref{eq:linphi} and~$V$ is \emph{quadratic}, the equilibrium measure
under conditioning is the unconditioned one translated in the direction
of~$v$. We provide a simple proof of this result in
Section~\ref{sec:quadratic}. We next turn to more general constraints in
Section~\ref{sec:LDP}, proving first an abstract Gibbs conditioning principle
in Section~\ref{sec:Gibbs}. When considering linear statistics, we prove in
Section~\ref{sec:linear} that conditioning~$P_n$ boils down to modifying the
confinement potential~$V$. In Section~\ref{sec:quadstats} we consider the case
of quadratic statistics, which amounts to modifying the interaction
kernel~$g$.

In order to validate our theoretical results and explore cases in which
explicit solutions are not available, we also propose a method for sampling
the law of~$Y_n$ for a fixed~$n$. Actually, sampling probability distributions
under constraint is a long standing problem in molecular dynamics and
computational statistics. Concerning molecular simulation, one can be
interested in fixing some degrees of freedom of a system like bond lengths, or
the value of a so-called reaction coordinate, typically for free energy
computations -- we refer \textit{e.g.}
to~\cite{darve2007thermodynamic,MR2681239} for more details.
An example of application in computational statistics is for instance
Approximate Bayesian Computations,
see~\cite{tavare1997inferring,marin2012approximate}. Based on the Hamiltonian
Monte Carlo (HMC) method used in~\cite{Chafai2018} for sampling Gibbs measures
associated to Coulomb and Log-gases, we describe and implement the generalized
Hamiltonian Monte Carlo algorithm proposed in~\cite{lelievre2018hybrid} for
sampling probability measures on submanifolds. The method is detailed in
Section~\ref{sec:HMC}, 
and the results presented in Section~\ref{sec:applications} are in agreement
with the theory developed in Section~\ref{sec:LDP}.

\subsection*{Notation}

We introduce some notation used throughout the paper. For all $d\geq1$ we
denote by $\left|x\right|=(x_1^2+\cdots+x_d^2)^{1/2}$ the Euclidean norm and
by $x\cdot y=x_1y_1+\cdots+x_d y_d$ the scalar product on~$\dR^d$. We denote
by~$\mathcal{P}(\dR^d)$ the set of probability measures on~$\dR^d$ and, for
all $p\geq1$, by $\mathcal{P}_p(\mathbb{R}^d)$ those probability measures
having finite $p$-moments in the sense that~$\left|\cdot\right|^p$ is
integrable. For any measure $\mu\in \mathcal{P}(\dR^d)$, the support of~$\mu$
is defined as $\mathrm{supp}(\mu)=\dR^d\setminus A$, where~$A$ is the largest
open set such that $\mu(A)=0$ (which may be empty). For all measurable
$f:\mathbb{R}^d\to\mathbb{R}$, we define
\[
  \|f\|_\infty=\mathop{\textrm{sup}}_{x\in\mathbb{R}^d}|f(x)|
  \quad\text{and}\quad
  \|f\|_{\mathrm{Lip}}=\sup_{x\neq y}\frac{|f(x)-f(y)|}{|x-y|}.
\]
We define the bounded-Lipschitz distance on~$\mathcal{P}(\mathbb{R}^d)$ by
\[
  \DBL(\mu,\nu)
  =\sup_{\substack{\|f\|_\infty\leq1\\\|f\|_\mathrm{Lip}\leq1}}\int_{\dR^d}
  f\,\mathrm{d}(\mu-\nu).
\]
For all $p\geq1$, we define the $p$-Wasserstein\footnote{Or Monge, or
  Kantorovich, or transportation distance.} distance on~$\mathcal{P}_p(\mathbb{R}^d)$ by
\begin{equation}
  \label{eq:inf_formulation_Wasserstein}
  \DWP(\mu,\nu)
  =\left(\inf_{\pi\in\Pi(\mu,\nu)}\iint_{\mathbb{R}^d\times\mathbb{R}^d}|x-y|^p\pi(\mathrm{d}x,\mathrm{d}y)\right)^{\frac{1}{p}},
\end{equation}
where~$\Pi(\mu,\nu)$ is the set of probability measures on the product
space~$\mathbb{R}^d\times\mathbb{R}^d$ with marginal distributions~$\mu$
and~$\nu$. For $p\geq 1$, the application $p\mapsto \DWP$ is monotonic.
Moreover, following \cite{villani2003topics}, the Kantorovich--Rubinstein
duality theorems state that
\begin{equation}
  \label{eq:kantorovitch}
  \DW(\mu,\nu)
  =\sup_{\|f\|_\mathrm{Lip}\leq1}\int_{\dR^d} f\,\mathrm{d}(\mu-\nu)
  \quad\text{and}\quad
  \DWP(\mu,\nu)^p
  =\sup_{\substack{f\in\mathrm{L}^1(\mu),\, g\in\mathrm{L}^1(\nu)\\
      f(x)\leq g(y)+|x-y|^p}}\left(\int_{\dR^d} f\,\mathrm{d}\mu-\int_{\dR^d}
  g\,\mathrm{d}\nu\right).
\end{equation}
For any $p\geq 1$, we say that a measurable function~$f$ is dominated by~$|x|^p$ when
\[
\| f\|_{\infty,p} = \mathop{\textrm{sup}}_{x\in\dR^d}\frac{|f(x)|}{1+|x|^p}<\infty.
\]

The $p$-Wasserstein topology is the one induced on~$\mathcal{P}_p(\dR^d)$
by~$\DWP$. If~$(\nu_n)_n$ is a sequence in~$\mathcal{P}(\mathbb{R}^d)$ then
$\lim_{n\to\infty}\DBL(\nu_n,\nu)=0$ if and only if
$\lim_{n\to\infty}\int f\mathrm{d}\nu_n=\int f\mathrm{d}\nu$ for all bounded
continuous $f:\mathbb{R}^d\to\mathbb{R}$. For all $p\geq1$ and all
sequence~${(\nu_n)}_n$ in~$\mathcal{P}_p(\mathbb{R}^d)$, we have
$\lim_{n\to\infty}\DWP(\nu_n,\nu)=0$ if and only if
$\lim_{n\to\infty}\DBL(\nu_n,\nu)=0$ and
$\lim_{n\to\infty}\int|x|^p\mathrm{d}\nu_n=\int|x|^p\mathrm{d}\nu$. In other
words~$\DBL$ metrizes weak convergence, while~$\DWP$ metrizes weak convergence
plus convergence of the $p$-moment, see~\cite{villani2003topics}. Moreover,
for any~$\mu,\nu\in\mathcal{P}(\dR^d)$, we denote by $\mu * \nu $ the
convolution of~$\nu$ with~$\mu$. This probability measure is defined by its
action over test functions~$\varphi$ through
\[
(\mu * \nu)(\varphi) = \iint_{\dR^d\times \dR^d}\varphi(x-y)\mu(\mathrm{d}y)\, \nu(\mathrm{d}x).
\]

We write $X\sim P$ to say that the random variable~$X$ has law~$P$, and
$X\overset{\mathrm{d}}{=}Y$ to say that the random variables~$X$ and~$Y$ have same law.
A sequence of random variables~$(X_n)_n$ is exchangeable if for any $n\geq 1$ and any
permutation~$\sigma$ of $\{1,\hdots,n\}$ it holds
$(X_1,\hdots,X_n) \overset{\mathrm{d}}{=} (X_{\sigma  (1)},\hdots,X_{\sigma(n)})$.

We say that a sequence of random variables~$(Z_n)_n$ taking values in a
metric space~$(\mathcal{Z},\mathrm{d})$ satisfies a large deviations principle
at speed~$(\beta_n)_n$ if, for any Borel set $A\subset \mathcal{Z}$, it holds
\begin{equation}
  \label{eq:abstractLDP}
  - \inf_{\mu\in\mathring{A}} I (\mu) \leq \liminf_{n\to +\infty}
  \frac{1}{\beta_n}\log \dP(Z_n\in A)
  \leq \limsup_{n\to +\infty} \frac{1}{\beta_n}\log \dP(Z_n\in A) \leq 
  - \inf_{\mu\in\overline{A}} I (\mu),
\end{equation}
where the interior and closure are taken with respect to the topology induced
by~$\mathrm{d}$, while $I:\mathcal{Z}\to[0,+\infty]$ is lower semicontinuous
and called the rate function. If~$I$ has compact level sets for the topology
induced by~$\mathrm{d}$, we say that~$I$ is a good rate function.

We finally recall some elements of potential theory. The interaction energy
\begin{equation}
  \label{eq:def_J}
  J(\mu) = \iint_{\dR^d\times \dR^d} g(x-y)\mu(\mathrm{d}x)\mu(\mathrm{d}y)
\end{equation}
is well defined and takes values in~$\dR\cup \{+\infty\}$ when~$\mu \in \mathcal{P}(\dR^d)$
satisfies~\eqref{eq:log_int_condition}. For $d \geq 3$, this is a consequence of the fact that
the kernel $g$ is positive. For $d=2$, we use the inequality $\log|x-y| \leq \log(1+|x|) + \log(1+|y|)$ to write
\[
  \begin{aligned}
  J(\mu) & = \iint_{|x-y|\leq 1}  -\log|x-y| \, \mu(\mathrm{d}x)\mu(\mathrm{d}y) + \iint_{|x-y|\geq 1} -\log|x-y| \, \mu(\mathrm{d}x)\mu(\mathrm{d}y) \\
	  & \geq \iint_{|x-y|\leq 1} -\log|x-y| \, \mu(\mathrm{d}x)\mu(\mathrm{d}y) -\iint_{|x-y|\geq 1}\left[ \log(1+|x|) + \log(1+|y|)\right] \mu(\mathrm{d}x)\mu(\mathrm{d}y).
  \end{aligned}
\]
The second integral on the last line is finite in view of the logarithmic integrability condition~\eqref{eq:log_int_condition}, while the first one is non-negative (although possibly infinite). Next, for $\mu,\nu \in \mathcal{P}(\dR^d)$ such that $J(\mu),J(\nu)<+\infty$, we define with some abuse of notation the interaction energy~$J(\mu,\nu)$ by polarization (as done in~\cite[Chapter~I]{MR0350027}):
\[
J(\mu,\nu) = \frac12 \Big( J(\mu)+J(\nu)-J(\mu-\nu)\Big), 
\]
where $J(\mu-\nu)$ is well defined with values in~$\dR\cup \{+\infty\}$ by~\cite[Theorem~1.16]{MR0350027}. 
Moreover, for any compact set $K\subset\dR^d$,~$J$ attains its infimum over
probability measures supported on~$K$. This value is called the
\emph{capacity} of the set~$K$~\cite[Chapter~II]{MR0350027}. A measurable
set~$A$ has \emph{positive capacity} if it contains a compact set~$K$ and a
measure~$\mu$ with $\mathrm{supp}(\mu)\subset K$ and such that
$J(\mu)<+\infty$. Otherwise, the set is said to have \emph{null capacity}. A
property is said to hold \emph{quasi-everywhere} if it is satisfied on a set
whose complementary has null capacity. Although inner and outer capacities
should be considered, we know these notions coincide for Borel sets
on~$\dR^d$, see~\cite[Theorem~2.8]{MR0350027}. Denoting by~$\Pc(\dR^d)$ the
set of compactly supported probability measures, in accordance with
\cite[Theorems~1.15 and~1.16]{MR0350027}, for any~$\mu,\nu\in\Pc(\dR^d)$ with
$J(\mu)<+\infty$ and $J(\nu)<+\infty$, it holds $J(\mu-\nu)=0$ if and only if
$\mu = \nu$.

\section{From conditioning to shifting: quadratic confinement with linear constraint}
\label{sec:quadratic}

This section is devoted to the particular case where
$V(x)= |x|^2$ and the constraint is chosen according
to~\eqref{eq:xinlinintro}-\eqref{eq:linphi}.
The following theorem states that this special case is exactly solvable: the
conditioning has the effect of a shift without deformation, due to a
remarkable factorization. The proof, presented in
Appendix~\ref{sec:proofquad}, is quite elegant. It is inspired from the
seemingly unrelated work~\cite{chafai-lehec}. The result by itself appears as
a special case of the general variational approach presented in
Section~\ref{sec:LDP}.

\begin{theorem}[From conditioning to shifting]
  \label{th:Hermite}
  Let $d,n\geq2$ and $V=\left|\cdot\right|^2$, so that~\eqref{eq:Zn} holds.
  Let $X_n=(X_{n,1},\ldots,X_{n,n})$ and~$P_n$ be as in~\eqref{eq:Pn}. Then
  the equilibrium measure~$\mus$ is the uniform law on the centered ball
  of~$\mathbb{R}^d$ of radius~$1$. Moreover, almost surely and for all $p\geq1$,
  it holds
  \begin{equation}\label{eq:shifting}
    \lim_{n\to\infty}\DWP\left(\frac{1}{n}\sum_{i=1}^n\delta_{X_{n,i}},\mus\right)
    =0,
  \end{equation}
  regardless of the way we define the random variables~$X_n$ on the same
  probability space. Now let $v\in \dR^d$ with $|v|=1$, $c\in\dR$, choose
  $\varphi(x)=x\cdot v - c$ and consider $Y_n=(Y_{n,1},\ldots,Y_{n,n})$ with
  \[
    Y_n\sim
    \mathrm{Law}\left(X_n \, \left|\,
        \frac{\varphi(X_{n,1})+\cdots+\varphi(X_{n,n})}{n}=0 \right. \right).
  \]
  Then
  \[
    Y_n\overset{\mathrm{d}}{=}
    X_n+\left(c-\frac{X_{n,1}+\cdots+X_{n,n}}{n}\cdot
      v\right)(v,\ldots,v).
  \]
  Moreover, denoting by $\mu^\varphi=\delta_{cv}*\mus$, we have that almost
  surely and for all $p\geq1$, it holds
  \begin{equation}
    \label{eq:quadcv}
    \lim_{n\to\infty}\DWP\left(\frac{1}{n}\sum_{i=1}^n\delta_{Y_{n,i}},\mu^\varphi\right)=0,
  \end{equation}
  regardless of the way we define the random variables~$Y_n$ on the same
  probability space.
\end{theorem}

The proof of Theorem \ref{th:Hermite} relies crucially on the quadratic nature
of the confinement potential, but remains valid whatever the pair interaction,
beyond Coulomb gases, as far as it is translation invariant. More general
linear projections can be used. Indeed, if we choose $\varphi(x) = p(x) - c$
where~$p$ is a linear projection over a subspace $E\subset\mathbb{R}^d$ of
dimension~$m$ and $c\in\dR^m$, the result still holds.

\begin{remark}[Ginibre random matrices]
  When $d=2$ and $\beta_n=(\beta/2)n^2$ for some $\beta>0$, the probability
  distribution~$P_n$ in Theorem~\ref{th:Hermite} is a Coulomb gas known as the
  \emph{$\beta$-Ginibre ensemble}.
  It appears also in the Laughlin fractional quantum Hall effect.
  The case $d=2$ and $\beta=2$ is even more remarkable and is known as the
  \emph{complex Ginibre ensemble}. More precisely, let~$M$ be an $n\times n$
  random matrix with independent and identically distributed complex Gaussian
  entries with independent real and imaginary parts of mean $0$ and variance
  $1/(2n)$. Its density is proportional to $\mathrm{e}^{-n\mathrm{Tr}(MM^*)}$.
  Its eigenvalues have law~$P_n$ with $d=2$, $\beta_n=n^2$, $V(x)=|x|^2$, see
  for instance~\cite[Chapter~15]{MR2641363}. Let $\mathbb{R}^2=\mathbb{C}$,
  $v\in\mathbb{R}^2$ with $|v|=1$ and $c\in \dR$. The assumptions of
  Theorem~\ref{th:Hermite} are satisfied and the constraint in terms of
  matrices reads $\mathrm{Tr}(M)\cdot v= nc$, where we identify
  again~$\mathbb{C}$ with~$\mathbb{R}^2$. More precisely~\eqref{eq:quadcv}
  holds and the conditioned equilibrium measure reads
    \[
    \mu^\varphi(dz)=
    \frac{\ind_{|z-cv|\leq1}}{\pi}\mathrm{d}z.
  \]
  Since the entries of~$M$ are independent, we have in particular 
  the decomposition $M=M-\mathrm{diag}(M)+\mathrm{diag}(M)$
  where $M-\mathrm{diag}(M)$ and $\mathrm{diag}(M)$ are independent. In this
  case we could probably deduce the desired result on $\mu^\varphi$
  from
  the Gibbs conditioning principle for independent Gaussian random variables
  to handle the diagonal part $\mathrm{diag}(M)$ conditioned on the value of
  $\mathrm{Tr}(M)=\mathrm{Tr}(\mathrm{diag}(M))$. Some numerical experiments
  are provided in Section~\ref{sec:applications}. Note that such fixed trace
  random matrix models appear in the Physics literature, see for instance
  \cite{MR3787269,MR2303267}.
\end{remark}

In practice, we would like to consider a non-quadratic confinement and
a non-linear constraint function~$\varphi$. The numerical applications presented
in Section~\ref{sec:applications} show a much wider range of behaviours than
shifting the equilibrium measure. It turns out that the conditioning mechanism
is an instance of the \emph{Gibbs conditioning principle} from large
deviations theory. The purpose of the next section is to provide proofs in
this direction, which allow to derive the conditioned equilibrium measure in
more general contexts, of which Theorem~\ref{th:Hermite} appears as a
particular case.

\section{A general conditioning framework}
\label{sec:LDP}

As is known from the seminal work of Ben Arous and
Guionnet~\cite{arous1997large}, large deviations theory provides a natural
framework to study the concentration of empirical measures of the spectrum of
random matrices and, beyond, of singularly interacting particles systems. We
refer in particular to~\cite{MR3262506,dupuis,berman2018coulomb} and
references therein for recent accounts. Since large deviations theory is
concerned with estimating probabilities of rare events, conditioning on such a
rare event is a natural direction to follow. This procedure is generally
referred to as \emph{Gibbs conditioning principle} or \emph{maximum entropy
  principle}. This principle is explained for instance
in~\cite[Section~6.3]{MR3309619} and~\cite[Section~7.3]{MR2571413}.

When no conditioning is considered, we know that under mild assumptions the
empirical measure associated to the Gibbs measure~\eqref{eq:Pn} satisfies a
LDP with rate function~$\cE$. When the random gas is considered under
conditioning on an appropriate rare event, the Gibbs conditioning principle
states that the resulting conditioned empirical measures concentrate on a
minimizer of~$\cE$ under constraint. Proofs of this fact in our context are
presented in Section~\ref{sec:Gibbs}. Next, Section~\ref{sec:linear} studies
the corresponding constrained minimization problem for linear statistics,
while Section~\ref{sec:quadstats} is concerned with quadratic statistics.

\subsection{Gibbs conditioning}
\label{sec:Gibbs}

The goal of this section is to present an abstract Gibbs conditioning
principle and apply it to the Coulomb gas model. Most works considered
hitherto Gibbs principles associated to Sanov's
theorem~\cite{MR2571413,MR3309619,leonard2002extension}, in other words in
absence of interaction, showing that by conditioning the empirical measure,
the resulting equilibrium measure minimizes the rate function under
constraint. The same strategy can actually be applied to any exchangeable
system satisfying a large deviations principle provided the conditioning set
is an $I$-continuity set, following for
instance~\cite{csiszar1984sanov},~\cite[Section~1.2]{MR2571413}
and~\cite[Section~5.3]{MR3309619}. This is the purpose of the next
theorem, which can be of independent interest, and whose proof is
postponed to Appendix~\ref{sec:proofLDP}.

\begin{theorem}[A Gibbs conditioning principle]\label{th:Gibbs}
  Suppose that ${(Z_n)}_n$ is a sequence of random variables taking values in
  a metric space~$(\mathcal{Z},\mathrm{d})$ satisfying a large deviations
  principle at speed~${(\beta_n)}_n$, and with good rate function~$I$.
  Consider a closed set~$B$ which is $I$-continuous in the sense that
  \begin{equation}
    \label{eq:Icontinuity}
    \inf_{\mathring{B}} I = \inf_B I < +\infty.
  \end{equation}
  Then, the set of minimizers
  \begin{equation}
    \label{eq:Imin}
  \mathscr{I}_{B} = \left\{z\in\mathcal{Z}:I(z) = \inf_B I\right\}
  \end{equation}
  is a non-empty closed subset of~$B$. Moreover, for any $\varepsilon >0$,
  setting
  \[
    A_{\varepsilon} = \big\{ z\in \mathcal{Z}:\mathrm{d}(z, \mathscr{I}_{B}) >
    \varepsilon \big\},
  \]
  there exists $c_{\varepsilon}>0$ such that
  \begin{equation}
    \label{eq:GibbsI}
    \limsup_{n\to +\infty} \frac{1}{\beta_n}\log \dP \left( Z_n \in A_{\varepsilon}\
      \Big| \ Z_n\in B\right) \leq - c_{\varepsilon}.
  \end{equation}
  In particular, if we assume that $\sum_n\mathrm{e}^{-c\beta_n}<+\infty$ for
  any $c>0$, and if we define a random variable
  $Z'_n\sim\mathrm{Law}(Z_n\mid Z_n\in B)$ for all~$n$, then almost surely
  \[
    \lim_{n\to\infty}\mathrm{d}(Z'_n,\mathscr{I}_B)=0,
  \]
  regardless of the way we define the~$Z'_n$'s on the same probability space.
  Finally, in the particular case where $\mathscr{I}_B=\{z_B\}$ is a singleton
  then almost surely it holds
  \[
    \lim_{n\to\infty}Z'_n=z_B=\min_BI.
  \]
\end{theorem}

In words,~\eqref{eq:GibbsI} shows that the variables~$Z_n$ conditioned on
being in~$B$ concentrate on a minimizer of~$I$ over~$B$. It is alluring to
consider a more general LDP for the conditioned empirical measure as
in~\cite{MR3392510}, but the arguments proposed in~\cite{MR3392510} do not fit
the case of the singular rate functional~\eqref{eq:cEbis}. Our strategy is
then to restrict to an $I$-continuity set~$B$
satisfying~\eqref{eq:Icontinuity}, which allows to use the lower bound of the
LDP, very similarly to~\cite{MR3309619}.

In order to use Theorem~\ref{th:Gibbs} in the Coulomb gas setting, we start by
recalling a LDP associated with the Coulomb gas model
(see~\cite{MR3262506,dupuis}). In order to consider unbounded constraints in
what follows, we make the following assumption.

\begin{assumption}[Growth condition]\label{as:Vq}
 There exist $a>0$, $R\in\dR$ and $q > 1$ such that
  \begin{equation*}
    \label{eq:growth}
    \forall\,x\in\dR^d,\quad  V(x)\geq a |x|^q -R.
  \end{equation*}
\end{assumption}

The above growth condition not only ensures that~$V$
satisfies~\eqref{eq:confinement}, but also allows to consider a finer topology
for the LDP, see~\cite[Theorem~1.8]{dupuis}. It could certainly be relaxed
under appropriate modifications. In particular, Assumption~\ref{as:Vq} shows
that~\cite[Assumption~C'1]{dupuis} is satisfied for any function of the
form~$|x|^p$ for $1< p<q$, so~\cite[Lemma~1.1]{dupuis} applies.
Assumption~\ref{as:Vq} thus has the following consequence. Consider an
exponent $ p\in(1,q)$. Then, under~$P_n$ defined
in~\eqref{eq:Pn}-\eqref{eq:betandiverge}, the empirical measure
\[
\mu_n = \frac{1}{n} \sum_{i=1}^n \delta_{x_i}
\]
satisfies a large deviations principle~\eqref{eq:abstractLDP} in the
$p$-Wasserstein topology at speed~$(\beta_n)_n$ and with the following good
rate function:
\[
\cE_{\star} = \cE - \inf_{\mathcal{P}_p(\dR^d)} \cE,
\quad \mathrm{with}\quad \inf_{\mathcal{P}_p(\dR^d)} \cE > - \infty,
\]
where~$\cE$ is defined in~\eqref{eq:cE}. The energy~$\cE$ has additional nice
properties, which we recall below for convenience.

\begin{proposition}[Properties of the electrostatic energy]
  \label{prop:cE}
  Let~$\mathcal{E}$ be as in \eqref{eq:cE}. Suppose that
  Assumption~\ref{as:Vq} holds, and take some $p\in(1,q)$. Then, denoting by
  $D_{\cE}= \{\mu\in\mathcal{P}(\dR^d):\cE(\mu)<+\infty \}$ the domain
  of~$\cE$, the following properties are satisfied:
  \begin{itemize}
  \item $D_{\cE}$ is convex and~$\cE$ is strictly convex on~$D_{\cE}$;
  \item $D_{\cE}\subset \mathcal{P}_p(\dR^d)$ and there exists a
    unique $\mus\in\Pc(\dR^d)$ such that
    \begin{equation}
      \label{eq:minimizers}
    \cE(\mus) = \inf_{\mathcal{P}(\dR^d)}\cE=\inf_{\mathcal{P}_p(\dR^d)}\cE
    =\inf_{\Pc(\dR^d)}\cE;
    \end{equation}
  \item the minimizer~$\mus\in\Pc(\dR^d)$ is unique and satisfies the
    Euler--Lagrange conditions (where $C_{\star} = \cE(\mus)$)
    \begin{equation}
      \label{eq:poteq}
    \left\{
    \begin{aligned}
      2 g*\mus + V = C_{\star},
      &\quad \mathrm{quasi-everywhere\ in\ supp}(\mus),\\
      2 g*\mus + V \geq C_{\star},
      &\quad \mathrm{quasi-everywhere}.
    \end{aligned}
    \right.
      \end{equation}
  \end{itemize}
\end{proposition}
The domain is not empty since it contains for instance measures with a smooth
density over a compact support. For convenience, we recall a proof of these
classical results in Appendix~\ref{sec:proofLDP}.

\begin{remark}[Going beyond Coulomb gases and convexity]
  The LDP presented here holds for a much larger range of models than the
  Coulomb gas setting, see for instance~\cite{dupuis}. However, the
  assumptions in~\cite{dupuis} do not ensure the convexity of the rate
  function~$\cE$, which poses problems when it comes to identifying the
  equilibrium measure -- we thus stick to this setting here. In practice, the
  convexity of~$\cE$ is derived from a Bochner-type positivity of the
  interaction potential, see~\cite{MR3262506}.
\end{remark}

We are now in position to apply Theorem~\ref{th:Gibbs} to the Coulomb gas
model.

\begin{corollary}[Gibbs conditioning for Coulomb gases]
  \label{co:Gibbs}
  Let~$\cE$ be as in \eqref{eq:cE}. Suppose that Assumption~\ref{as:Vq} holds
  and take some $p\in(1,q)$. Consider a closed
  set~$B\subset \mathcal{P}_p(\dR^d)$ such that
  \begin{equation}
    \label{eq:Gibbscond}
    \inf_{\mathring{B}} \cE = \inf_B \cE < +\infty,
  \end{equation}
  where the interior is taken with respect to the $p$-Wasserstein topology.
  Then the set of minimizers
  \begin{equation}\label{eq:Emin}
    \mathscr{E}_{B}
    = \left\{\mu\in\mathcal{P}(\dR^d):\cE(\mu) = \inf_B \cE \right\}
  \end{equation}
  is a non-empty closed subset of~$B$. Moreover, if $X_n\sim P_n$ is as
  in~\eqref{eq:Pn}, and if $Y_n=(Y_{n,1},\hdots, Y_{n,n})$ is such that
  \[
  Y_n\sim \mathrm{Law}\big(X_n\, \big|\, \mu_n \in B \big),
  \quad \mathrm{with}\quad
  \mu_n = \frac{1}{n} \sum_{i=1}^n \delta_{X_{n,i}},
  \]
  then almost surely it holds 
  \[
    \lim_{n\to\infty}
    \DWP\left( \frac{1}{n} \sum_{i=1}^n \delta_{Y_{n,i}}, \mathscr{E}_B\right)
    =0,
  \]
  regardless of the way we define the random variables~$Y_n$ on the same
  probability space.
\end{corollary}

The proof of Corollary~\ref{co:Gibbs}, which can be found in
Appendix~\ref{sec:proofLDP}, is an instance of the Gibbs conditioning
principle of Theorem~\ref{th:Gibbs}: the conditioned empirical measure
concentrates almost surely on a minimizer of~$\cE$ over~$B$.

Next, rather than aiming at the greatest generality, we consider the case of
linear and quadratic statistics constraints, for which $I$-continuity can be
proved and the resulting equilibrium measure can be identified in terms of a
modified version of~\eqref{eq:poteq}.

\subsection{Linear statistics}
\label{sec:linear}

As explained in the introduction, the case of \emph{linear statistics} is of
particular importance. This motivates focusing first on conditioning
sets~$B \subset \mathcal{P}_p(\dR^d)$ of the form
\begin{equation}\label{eq:setB}
  B = \big\{ \nu\in\mathcal{P}_p(\dR^d):\nu(\varphi) \leq 0\big\},
\end{equation}
for some measurable function~$\varphi:\mathbb{R}^d\to\mathbb{R}$. This kind of
constraint was studied for example
in~\cite{MR3264829,ghosh-nishry,MR3825948,PhysRevLett.103.220603,PhysRevE.83.041105}.
In particular, the Ginibre case with $\varphi = \mathds{1}_U - c$ for a
measurable set $U\subset \dR^2$ and $c\in \dR$ is considered
in~\cite{MR3264829}. The choice $\varphi(x) = c - x\cdot v $ for $v\in\dR^d$
has been treated in Section~\ref{sec:quadratic} for the related equality
constraint. We consider here more general potentials~$V$ and constraint
functions~$\varphi$. The next assumption on~$\varphi$ ensures that~$B$ is
suitable for conditioning.

\begin{assumption}\ %
  \label{as:intB}
  \begin{itemize}
  \item Assumption~\ref{as:Vq} holds for some $q>1$;
  \item $\|\varphi\|_{\mathrm{Lip}}<+\infty$ (and thus
    $\|\varphi\|_{\infty,p}<+\infty$ for all $p\geq1$);
  \item there exists $\mu_-\in D_{\cE}$ such that $\mu_-(\varphi)<0$;
  \item there exists $\mu_+\in D_{\cE}$ such that $\mu_+(\varphi)>0$.
  \end{itemize}
\end{assumption}
 
The existence of~$\mu_-$ means that the set~$B$ has non empty interior, while
that of~$\mu_+$ implies that $B\neq \mathcal{P}_p(\dR^d)$, so that the
constraint is not trivial. Since the Gibbs principle relies on~$B$ being an
$I$-continuity set, we provide a fine analysis of the minimization of~$\cE$
over the set~$B$ defined in~\eqref{eq:setB}. We prove in particular that the
minimizer is unique with compact support, and we characterize it through an
integral equation similar to~\eqref{eq:poteq} with an additional Lagrange
multiplier. The proof of this result is presented in
Appendix~\ref{sec:prooflincond}. 

\begin{theorem}[Variational characterization]\label{th:lincond}
  Let $\mus\in\Pc(\dR^d)$ be the unconstrained equilibrium measure as in
  Proposition~\ref{prop:cE}, and let~$B$ be the set defined in~\eqref{eq:setB}.
  Suppose that Assumption~\ref{as:intB} hold for some $q>1$, and choose~$p\in(1,q)$.
  Then~$B$ is closed in the $p$-Wasserstein topology and
    \begin{equation}
      \label{eq:Icontlin}
  \inf_{\mathring{B}}\cE = \inf_B \cE<+\infty.
  \end{equation}
  Moreover
  \[
  \mathscr{E}_B=\left\{\mu\in\mathcal{P}(\mathbb{R}^d):\cE(\mu)=\inf_B\cE\right\}
  =\left\{\mu^\varphi\right\},
  \]
  where $\mu^\varphi$ has compact support and is solution to, for some
  $\alpha\geq 0$,
  \begin{equation}
    \label{eq:muphi}
    \left\{
  \begin{aligned}
    2 g * \mu^{\varphi} + V + \alpha\varphi = C_{\varphi}, &\quad \mathrm{quasi-everywhere\ in}\
    \mathrm{supp} (\mu^{\varphi}),\\
    2 g * \mu^{\varphi} + V + \alpha\varphi \geq C_{\varphi}, &\quad \mathrm{quasi-everywhere},
  \end{aligned}
  \right.
  \end{equation}
  with $C_{\varphi}= \cE(\mu^\varphi)$. Finally, one of the two following
  conditions holds:
  \begin{itemize}
  \item $\mus\in B$ and $\alpha = 0$;
  \item $\mus\notin B$, in which case $\mu^{\varphi}(\varphi)=0$ and
    $\alpha >0$. In other words, the constraint is saturated and the Lagrange
    multiplier is active.
  \end{itemize}
\end{theorem}

We now have the following consequence of Corollary~\ref{co:Gibbs} and
Theorem~\ref{th:lincond}.

\begin{corollary}[From conditioning to confinement deformation]
  \label{cor:linGibbs}
  Suppose that Assumption~\ref{as:intB} holds for some~$q>1$, and choose~$p\in(1,q)$.
  Consider a Coulomb gas $X_n=(X_{n,1},\hdots,X_{n,n})\sim P_n$ as
  in~\eqref{eq:Pn}. Introduce $Y_n=(Y_{n,1},\hdots,Y_{n,n})$ with law given by
  \[
    Y_n\sim
    \mathrm{Law}\left(X_n\,\left|\, \frac{1}{n}\sum_{i=1}^n
        \varphi(X_{n,i})\leq 0\right.\right).
  \]
  Let $\mu^\varphi$ be as in Theorem~\ref{th:lincond}. Then almost surely it
  holds
  \begin{equation}
    \label{eq:Ynlincv}
    \lim_{n\to\infty}\DWP\left(
    \frac{1}{n}\sum_{i=1}^n\delta_{Y_{n,i}},\mu^\varphi\right)=0,
  \end{equation}
  regardless of the way we define the random variables~$Y_n$ on the same
  probability space.
\end{corollary}

Theorem~\ref{th:lincond} and Corollary~\ref{cor:linGibbs} show that
conditioning on a linear statistics is equivalent to changing the confinement
potential~$V$ into $V+ \alpha \varphi$ where $\alpha\geq 0$ is a constant
determined by the constraint. If~$\mus\in B$, $\alpha=0$ and the conditioning
produces no effect. Note also that the global Lipschitz condition on~$\varphi$
in Assumption~\ref{as:intB} could possibly be relaxed. For instance, if
$\|\varphi\|_{\infty,p}<+\infty$ for some $p\in(1,q)$ we expect that a
minimizing measure has compact support, and that assuming~$\varphi$ locally
Lipschitz suffices to prove Theorem~\ref{th:lincond}. We leave these
refinements to further studies.

\begin{remark}[Equality constraints]
  When considering conditioning principles, one is often interested in
  equality constraints. It is not obvious at first sight to consider a set~$B$
  defined by an equality constraint, since its interior may well be empty. A
  common strategy is to use a limiting procedure by introducing nested
  sets~\cite{MR2571413}. This is unnecessary here since we observe in
  Theorem~\ref{th:lincond} that either the equilibrium measure lies in~$B$,
  or the constraint is saturated.
\end{remark}

\begin{remark}[Projection]
  The conditioned equilibrium measure~$\mu^\varphi$ can be interpreted as an
  instance of \emph{entropic projection}. These projections have been studied
  for a long time in the context of the Sanov theorem, in other words
  independent particles or equivalently product measures without interaction
  at all, see~\cite{leonard2010entropic} and the references therein.
  Theorem~\ref{th:lincond} is therefore a precise study of such a projection
  in the context of Coulomb gases under linear statistics constraints, where
  the entropy is replaced by the electrostatic energy~$\cE$. These remarks
  also apply to Section~\ref{sec:quadstats}.
\end{remark}

\begin{remark}[Formula for constrained equilibrium measure under regularity
  assumptions]
  \label{rm:solutionlinear}
  Suppose that Assumption \ref{as:intB} holds and that~$V$ and~$\varphi$ have
  Lipschitz continuous derivatives. Then the conditioned equilibrium
  measure~$\mu^{\varphi}$ which appears in Theorem~\ref{th:lincond} and
  in Corollary~\ref{cor:linGibbs} satisfies
  \begin{equation}\label{eq:muphisol}    
    \left\{
    \begin{aligned}
      &\mu^{\varphi}
      =\frac{\Delta V + \alpha\Delta \varphi}{2c_d}, 
      &\text{almost everywhere in}\ \mathrm{supp}
      (\mu^{\varphi}),\\
      & \mu^{\varphi} =0,
      &\text{almost everywhere outside}\ \mathrm{supp}  (\mu^{\varphi}).
    \end{aligned}
    \right.
  \end{equation}
  Indeed, it suffices to apply the Laplacian to both sides of~\eqref{eq:muphi}
  and use~\eqref{eq:poisson}; see for
  example~\cite[Proposition~2.22]{MR3309890} for the technical details. We
  mention that a density is non-negative, hence
  \[
  \mathrm{supp}(\mu^\varphi)\subset\big\{ x\in\dR^d \ \big|\ \Delta V
  + \alpha \Delta\varphi \geq 0\big\}.
  \]
  The constraint may therefore significantly change the support of the
  equilibrium measure.
\end{remark}

It is now possible to come back to the translation phenomenon described in
Section~\ref{sec:quadratic} through the energetic approach considered in this
section.

\begin{proof}[Alternative proof of Theorem~\ref{th:Hermite}]
  Using~\eqref{eq:muphisol} under the assumptions of Theorem~\ref{th:Hermite},
  we have $\Delta V= 2d$ and $\Delta \varphi = 0$, so that~$\mu^{\varphi}$ is
  constant and equal to~$d/c_d$ on its support. It then remains to show that
  this support is indeed a ball of correct center and radius. For this, we
  observe that, since $|v|=1$,
  \[
    V(x) +\alpha \varphi(x) = |x|^2 + \alpha( c - x\cdot v) %
    = \left|x - \frac{\alpha v}{2}\right|^2 +
    \frac{\alpha^2}{4} + \alpha c,
  \]
  so that the effective confining potential is quadratic with variance~$1/2$
  and center $x_0 = \alpha v /2$. By radial symmetry around~$x_0$,
  $\mu^{\varphi}$ must be a uniform distribution on a ball~$B(x_0,r)$ centered
  at~$x_0$ with radius $r>0$. In order to find the value of~$\alpha$, we write
  the constraint
  \[
  |B(x_0,r)|^{-1}\int_{B(x_0,r)} x\cdot v \,\mathrm{d}x = c.
  \]
  The left hand side of the above equation reads, by symmetry,
  \[
  |B(x_0,r)|^{-1} \int_{B(x_0,r)} (x-x_0)\cdot v \,\mathrm{d}x
  +x_0\cdot v   = |B(x_0,r)|^{-1} \int_{B(0,r)} x\cdot v \,\mathrm{d}x +  x_0\cdot v
  = x_0 \cdot v.
  \]
  Since~$x_0 = \alpha v /2$ and $|v|^2=1$ we obtain
  \[
  \alpha = 2c,
  \]
  which leads to $x_0 =cv$. Finally, the value of~$\mu^\varphi$ over its
  support is~$d/c_d$, where~$c_d$ is the surface of the sphere in
  dimension~$d$. Since the volume of the sphere of radius~$r$ is equal
  to~$rc_d/d$, we obtain that $r=1$ and we reach the conclusion of
  Theorem~\ref{th:Hermite}.
\end{proof}

\subsection{Quadratic statistics}
\label{sec:quadstats}

Once the linear statistics case has been studied, it is natural to turn to
more general constraints. Considering second order statistics is a first step
in this direction, which motivates considering sets of the form, for $p> 1$,
\begin{equation}\label{eq:Bquad}
  B = \big\{ \nu\in\mathcal{P}_p(\dR^d):Q(\nu) \leq 0\big\},
\end{equation}
where $Q$ is the ``quadratic form''
\begin{equation}\label{eq:Qdef}
  Q:\mu\in\mathcal{P}_p(\dR^d)%
  \mapsto \iint_{\dR^d\times \dR^d} \psi(x,y)\mu(\mathrm{d}x)\mu(\mathrm{d}y),
\end{equation}
and $\psi:\dR^d\times\dR^d\to\dR$ is a prescribed function. For
any~$\mu\in\mathcal{P}(\dR^d)$, we denote by
\[
U_\mu^\psi:x\in\dR^d \mapsto \int_{\dR^d} \psi(x,y)\mu(\mathrm{d}y)
\]
the ``potential'' generated by~$\mu$ for the interaction~$\psi$, whenever this
makes sense. We now make some assumptions on the interaction~$\psi$ for the
functional~$Q$ to define an $I$-continuity set~$B$ in~\eqref{eq:Bquad}.

\begin{assumption}\ %
  \label{as:psi}
  \begin{itemize}
  \item Assumption~\ref{as:Vq} holds for some $q>1$;
  \item there is~$C_{\mathrm{Lip}}>0$
    such that, for any $p \geq 1$ and $\mu\in\mathcal{P}_p(\dR^d)$,
    \begin{equation}
      \label{eq:psicond}
      \big\| U_\mu^\psi \big\|_{\mathrm{Lip}}\leq C_{\mathrm{Lip}}
    \end{equation}
    (and thus $\big\| U_\mu^\psi\big\|_{\infty,p} < +\infty$ for all
    $p\geq1$);
  \item $\psi$ is symmetric, \textit{i.e.} $\psi(x,y)=\psi(y,x)$ for all
    $x,y\in\mathbb{R}^d$;
  \item $Q$ is convex;
  \item there exists $\mu_-\in D_{\cE}$ such that $Q(\mu_-)<0$;
  \item there exists $\mu_+\in D_{\cE}$ such that $Q(\mu_+)>0$.
  \end{itemize}
\end{assumption}

Before turning to the minimization under constraint, let us present a class of
functions~$\psi$ for which~\eqref{eq:psicond} holds. 
\begin{proposition}[Sufficient condition for
  \eqref{eq:psicond}]\label{prop:suffquad}
  Assume that $\psi(x,y) = \phi(x - y)$ for a function $\phi:\dR^d\to\dR$
  satisfying $\|\phi\|_{\mathrm{Lip}}<+\infty$ (and thus
  $\|\phi\|_{\infty,p}<+\infty$ for any $p \geq 1$). Then \eqref{eq:psicond} holds with
  $C_{\mathrm{Lip}}=\|\phi\|_{\mathrm{Lip}}$.
\end{proposition}

\begin{proof}
  Fix~$p \geq 1$. For all $\nu\in\mathcal{P}_p(\dR^d)$
  and $x,x'\in\dR^d$, it holds
  \[
  \big| U_\nu^\psi(x) - U_\nu^\psi(x')\big|  \leq \int_{\dR^d}
  | \phi(x-y) - \phi(x' - y)|\nu(\mathrm{d}y)
    \leq\|\phi\|_{\mathrm{Lip}} \int_{\dR^d} |x - y - (x' - y)| \nu(\mathrm{d}y)
  = \|\phi\|_{\mathrm{Lip}}|x - x'|.
  \]
  We thus obtain~\eqref{eq:psicond}
  with~$C_{\mathrm{Lip}} = \|\phi\|_{\mathrm{Lip}}$.
\end{proof}

In addition to the regularity ensured by Proposition~\ref{prop:suffquad}, the
convexity of~$Q$ is an important part of Assumption~\ref{as:psi}. Conditions
for this convexity to hold for an interaction of the form
$\psi(x,y)=\phi(x-y)$, which is related to Bochner-type positivity, are
discussed at length in~\cite{berg1984harmonic,koldobsky2005fourier,MR3262506}.
In particular, the choice~$\psi = g$ where~$g$ is defined
in~\eqref{eq:poisson} leads to a convex~$Q$. Assumption~\ref{as:psi} then
provides a result similar to that of Theorem~\ref{th:lincond}, now leading to a
deformation of the interaction energy. The proof can be found in
Appendix~\ref{sec:proofquadstats}.

\begin{theorem}[Quadratic constraint]\label{th:quadcond}
  Let $\mus\in\cP(\dR^d)$ be the unconstrained equilibrium measure defined in
  Proposition~\ref{prop:cE}. Suppose that Assumption~\ref{as:psi} holds for
  some $q>1$, fix $p\in(1,q)$ and let~$B$ be the set defined in~\eqref{eq:Bquad}.
  Then~$B$ is closed in the $p$-Wasserstein topology and
    \begin{equation}
      \label{eq:Icontquad}
  \inf_{\mathring{B}}\cE = \inf_B \cE<+\infty.
  \end{equation}
  Moreover
  \[
    \mathscr{E}_B =\left\{\mu\in\mathcal{P}(\mathbb{R}^d):
    \cE(\mu)=\inf_B\cE\right\}=\left\{\mu^\psi\right\},
  \]
  where~$\mu^\psi$ has compact support and is solution to, for some
  $\alpha\geq 0$,
  \begin{equation}
    \label{eq:mupsi}
    \left\{
  \begin{aligned}
    2g* \mu^{\psi} + 2\alpha U_{\mu^\psi}^\psi + V
    = C_{\psi}, &\quad \mathrm{quasi-everywhere\ in}\
    \mathrm{supp} (\mu^{\psi}),\\
     2g* \mu^{\psi} + 2\alpha U_{\mu^\psi}^\psi + V \geq C_{\psi}, &\quad \mathrm{quasi-everywhere},
  \end{aligned}
  \right.
  \end{equation}
  with $C_{\psi}= \cE(\mu^\psi)$. 
  Finally, one of the two following conditions holds:
  \begin{itemize}
  \item $\mus\in B$ and $\alpha = 0$;
  \item $\mus\notin B$, in which case $Q(\mu^{\psi})=0$ and $\alpha >0$.
  \end{itemize}
\end{theorem}

The quadratic constraint leads to a change in the interaction contrarily to
the linear situation, which led to a change of confinement. From
Theorem~\ref{th:quadcond}, we obtain the following result.

\begin{corollary}[From conditioning to interaction deformation]
  \label{cor:quadGibbs}
  Suppose that Assumption~\ref{as:psi} holds for some $q> 1$, and fix $p\in(1,q)$.
  Consider a Coulomb gas $X_n=(X_{n,1},\hdots,X_{n,n})\sim P_n$ as
  in~\eqref{eq:Pn}, and $Y_n=(Y_{n,1},\hdots,Y_{n,n})$ with law given by
  \[
    Y_n\sim\mathrm{Law}\left(X_n\,\left|\,
    \frac{1}{n^2}\sum_{i,j=1}^n \psi(X_{n,i}-X_{n,j})\leq 0 \right.\right).
  \]
  Let~$\mu^\psi$ be the conditioned equilibrium measure as in
  Theorem~\ref{th:quadcond}. Then, almost surely, it holds
  \begin{equation}
    \label{eq:Ynquadcv}
    \lim_{n\to\infty}
    \DWP\left(
    \frac{1}{n}\sum_{i=1}^n\delta_{Y_{n,i}},
    \mu^\psi
    \right)
    =0,
  \end{equation}
  regardless of the way we define the random variables~$Y_n$ on the same
  probability space.
\end{corollary}

\begin{remark}[Higher order constraints, convexity and regularity]
  From the proof of Theorem~\ref{th:quadcond}, the Gibbs principle holds for a
  set~$B$ of the form~\eqref{eq:Bquad} for~$Q$ convex and lower
  semicontinuous. However, we would not be able to say much on the solution in
  such an abstract setting. In particular, higher order statistics could be
  considered, leading to higher order convolutions, but checking the convexity
  of the associated functional becomes less convenient. By lack of
  applications in mind, we do not consider these higher order constraints here.
\end{remark}

\begin{remark}[Formula for the constrained equilibrium measure under regularity
  assumptions]
  \label{rm:solutionquad}
  Suppose that Assumption \ref{as:psi} holds,~$V$ has Lipschitz
  continuous derivatives and~$\psi$ is~$\mathcal{C}^2(\dR^d)$. Then the
  conditioned equilibrium measure~$\mu^{\psi}$ that appears in
  Theorem~\ref{th:quadcond} and in Corollary \ref{cor:quadGibbs} satisfies
  \begin{equation}
    \label{eq:mupsisol}
    \left\{
    \begin{aligned}  
      &    c_d \mu^{\psi} -\alpha \int_{\dR^d} \Delta\psi(\cdot,y)\mu^\psi(\mathrm{d}y)
      =   \frac{\Delta V}{2} ,&
      \text{almost everywhere in}\ \mathrm{supp}
      (\mu^{\psi}),\\
      &   \mu^{\psi} = 0,&  \text{almost everywhere outside}\ \mathrm{supp}
      (\mu^{\psi}).
  \end{aligned}
    \right.
  \end{equation}
  Indeed,~\eqref{eq:mupsisol} follows by applying the Laplacian on both sides
  of~\eqref{eq:mupsi}. Note that the expression~\eqref{eq:mupsisol} is not
  explicit as in the linear constraint case of Remark~\ref{rm:solutionlinear}
  because we are not able to invert the convolution associated to~$\Delta\psi$
  in general.
\end{remark}

\section{Numerical illustration}
\label{sec:numerics}

In this section we consider the problem of sampling from conditioned distributions of the form
\begin{equation}
  \label{eq:constxi}
 \mathrm{Law}\big( X_n \,\big|\, \xi_n(X_n) = 0 \big),
\end{equation}
where $\xi_n:(\dR^d)^n\to \dR^m$ for some $m\geq 1$, and~$X_n$ is distributed
according to~$P_n$ defined in~\eqref{eq:Pn} for~$n$ fixed. We drop the
index~$n$ on~$\xi_n$ in what follows to shorten the notation, and consider
constraints taking values in~$\dR^m$ for generality. Note that we consider
equality rather than inequality constraints since we have seen in
Sections~\ref{sec:linear} and~\ref{sec:quadstats} that inequality constraints
are either satisfied by the equilibrium measure or saturated.

The first contribution of this section is to propose in Section~\ref{sec:HMC}
an algorithm for sampling from~\eqref{eq:constxi}. In a second step, we
present in Section~\ref{sec:applications} some numerical applications, where
we illustrate the predictions of Sections~\ref{sec:quadratic}
and~\ref{sec:LDP}. This is also the opportunity to explore conjectures which
are not proved in the present paper.

\subsection{Description of the algorithm}
\label{sec:HMC}

The description of the constrained Hamiltonian Monte Carlo algorithm used for
sampling follows several steps. We first make precise the structure of the
measure~\eqref{eq:constxi} in Section~\ref{sec:dirac}.
Section~\ref{sec:constlangevin} next introduces a constrained Langevin
dynamics used for sampling, while Section~\ref{sec:langevindiscretization}
gives the details of the numerical integration.

\subsubsection{Dirac and Lebesgue measures on submanifolds}
\label{sec:dirac}

Let us first describe more precisely the structure of the constrained
measure~\eqref{eq:constxi} by introducing the submanifold~$\cM_z$ associated
with the $z$-level set of~$\xi$ for $z\in\dR^m$, namely
\begin{equation}
  \label{eq:manifold}
\cM_z = \big\{ x\in(\dR^d)^n:\xi(x) = z\big\}.
\end{equation}
We use the shorthand notation $\cM = \cM_0$. To define the conditioned
measure, we rely on the following disintegration of (Lebesgue) measure
formula: for any bounded continuous test function~$\varphi$,
\begin{equation}\label{eq:conditionedmeasure}
\int_{(\dR^d)^n} \varphi(x)\,\mathrm{d}x = \int_{\dR^m}\int_{\cM_z}\varphi(x)\delta_{\xi(x)-z}
(\mathrm{d}x)\,\mathrm{d}z.
\end{equation}
This defines for any~$z\in\dR^m$ the conditioned
measure~$\delta_{\xi(x)-z}(dx)$, see~\cite[Section~2.3.2]{MR2945148}.
Since~$P_n$ is given by~\eqref{eq:Pn}, the constrained
measure~\eqref{eq:constxi} can be written with the conditioned
measure~$\delta_{\xi(x)}(dx)$ associated with~$\cM$ as: for any bounded
continuous~$\varphi$,
\begin{equation}
  \label{eq:constaverage}
  \bE\big[\varphi(X_n)\,\big|\,\xi_n(X_n) = 0 \big] =\frac{1}{Z_n^\xi} \int_{\cM} \varphi(x)
  \, \mathrm{e}^{-\beta_n H_n(x)}\delta_{\xi(x)}(\mathrm{d}x),
\end{equation}
where~$Z_n^\xi$ is a normalizing constant. The measure of interest
is therefore
\begin{equation}
  \label{eq:Pnxi}
  P_n^\xi(\mathrm{d}x) = \frac{\mathrm{e}^{- \beta_n H_n(x) } }{Z_n^\xi}
  \delta_{\xi(x)}(\mathrm{d}x).
\end{equation}
In order to obtain a better understanding of~\eqref{eq:Pnxi}, we relate the
conditioned measure proportional to~$\delta_{\xi(x)}(\mathrm{d}x)$ to the
volume measure induced on the submanifold~$\cM$ by the canonical Euclidean
metric on $(\mathbb{R}^{d})^n$, 
which we denote by~$\sigma_{\cM}(\mathrm{d}x)$, see
\cite[Section~3.2.1.1]{MR2681239}. We use to this end the co-area
formula~\cite{ambrosio2000functions,evans2018measure,MR2945148}. We denote by
$\nabla\xi = (\nabla \xi_1,\hdots,\nabla \xi_m)\in\dR^{dn\times m}$ and
introduce the Gram matrix:
\begin{equation}
  \label{eq:gram}
  G(x) = \nabla \xi(x)^T \nabla \xi(x) \in \dR^{m\times m},
  \end{equation}
  where the superscript~$T$ denotes matrix transposition. In what follows, we
  assume that the Gram matrix~\eqref{eq:gram} is non-degenerate in the sense
  that~$G(x)$ is invertible for~$x$ in a neighborhood of~$\cM$
  (see~\cite[Proposition~2.1]{MR2945148}).

\begin{proposition}\label{prop:coarea}
  The measures~$\delta_{\xi(x)}(\mathrm{d}x)$ and~$\sigma_{\cM}(\mathrm{d}x)$
  are related by
  \begin{equation}
    \label{eq:deltaxi}
  \delta_{\xi(x)}(\mathrm{d}x) = |\mathrm{det}\, G(x)|^{-\frac{1}{2}}\sigma_{\cM}(\mathrm{d}x).
  \end{equation}
  In particular it holds
  \begin{equation}
    \label{eq:Pnxileb}
  P_n^\xi(\mathrm{d}x) = \frac{\mathrm{e}^{- \beta_n H_n^\xi(x) } }{Z_n^\xi}
  \sigma_{\cM}(\mathrm{d}x),
  \end{equation}
  where
\begin{equation}
    \label{eq:Hcorrect}
    H_n^\xi(x) = H_n(x) + U_n(x),\quad U_n(x)= - \frac{1}{2\beta_n} \log |\mathrm{det}\, G(x)|. 
  \end{equation}
\end{proposition}

\begin{remark}[Parametrization invariance]
  \label{rem:parametrization}
  It seems at first sight that the definition of the conditioned measure
  in~\eqref{eq:conditionedmeasure} depends on the choice of parametrization
  of~$\xi$, but it does not. To illustrate this point, we consider for
  simplicity that $m=1$ and~$\cM = \big\{ x\in(\dR^d)^n:\xi(x) = 0\big\}$.
  
  First, the induced volume measure on~$\cM$ does not depend on the
  parametrization of~$\cM$. Consider next a smooth function $F:\dR\to\dR$ such
  that $F(0) =0$ and $F'(0)\neq 0$, and the change of
  parametrization~$\cM = \big\{ x\in(\dR^d)^n:F\big(\xi(x)\big) = 0\big\}$.
  The gradient of the constraint at $x\in(\dR^d)^n$ is then
  $\nabla F(\xi(x)) = F'(\xi(x))\nabla \xi(x)$. Since~$\xi(x)=0$ for
  $x\in\cM$, the right hand side of~\eqref{eq:deltaxi} is changed only by a
  multiplicative factor~$|F'(0)| \neq 0$. Therefore, the conditioned
  probability measure~\eqref{eq:Pnxi} is left unchanged.
\end{remark}

The aim is therefore to sample from~\eqref{eq:Pnxileb}, which is not an easy
task except in very particular situations, like the one studied in
Section~\ref{sec:quadratic}. As already noted in the introduction, for the rare events that we consider, it would
not be efficient to use a direct approach based on rejection sampling as what
is done in~\cite{ghosh-nishry,MR3825948,PhysRevE.83.041105} for inequality
constraints. We thus resort to sampling on submanifolds.

For sampling measures on submanifolds, a naive penalization of the constraint
is not a good idea in general, since the dynamics used to sample it (such
as~\eqref{eq:langevin} below) are difficult to integrate because of the
stiffness of the penalized energy. Moreover, our problem is made harder by the
singularity of the pair interaction in the Hamiltonian~\eqref{eq:Hn}. It is
known that Hybrid Monte Carlo schemes (relying on a second order
discretization of an underdamped Langevin dynamics with a Metropolis--Hastings
acceptance rule) provide efficient methods for sampling such probability
distributions, see~\cite{Chafai2018} and references therein.
However, Metropolis--Hastings schemes crucially rely on the reversibility of the proposal.
An issue when
combining a Metropolis--Hastings rule with a projection on a submanifold is
that reversibility may be lost, which introduces a bias. A recent strategy has
been to introduce a reversibility check in addition to the standard
acception-rejection rule, which makes the HMC scheme under constraint
reversible~\cite{zappa2018monte,lelievre2018hybrid}. Note
that~\cite{zhang2017ergodic} proposes an interesting alternative to the scheme
used here, which is however not compatible with a Metropolis selection
procedure in its current form. We thus present the algorithm as written
in~\cite{lelievre2018hybrid}, with some simplifications and adaptations to our
context, for which we introduce next the constrained Langevin dynamics.

\subsubsection{Constrained Langevin dynamics}
\label{sec:constlangevin}

We define here an underdamped Langevin dynamics over the submanifold~$\cM$, whose invariant measure
has a marginal in position which coincides with~\eqref{eq:Pnxileb}. We motivate using
this dynamics by first considering the problem of sampling from the unconstrained
measure~$P_n$. For a given $\gamma > 0$, we define
\begin{equation}
  \label{eq:langevin}
  \left\{
  \begin{aligned}
     \mathrm{d}X_t & = Y_t \,\mathrm{d}t,\\
     \mathrm{d}Y_t & = -\nabla H_n(X_t)\,\mathrm{d}t -\gamma Y_t\, \mathrm{d}t
    + \sqrt{\frac{2\gamma}{\beta_n}}\,\mathrm{d}W_t,
  \end{aligned}
  \right.
\end{equation}
where~$(W_t)_{t\geq 0}$ is a $dn$-dimensional Wiener process. In this
dynamics,~$(X_t)_{t\geq 0}$ stands for a position, while~$(Y_t)_{t\geq 0}$
represents a momentum variable. Let us mention that the long time convergence
of the law of this process towards~$P_n$ (a difficult problem due to the
singularity of the Hamiltonian) can be proved through Lyapunov function
techniques, see~\cite{lu2019geometric} for a recent account. In practice, the
singularity of~$g$ also makes the numerical integration of~\eqref{eq:langevin}
difficult, and a Metropolis--Hastings selection rule can be used to stabilize
the numerical discretization, see~\cite{Chafai2018} and references therein.
The algorithm described below makes precise how to adapt this strategy to
sample measures constrained to the submanifold~$\cM$.

Since we aim at sampling from~\eqref{eq:Pnxileb}, it is natural to consider
the dynamics~\eqref{eq:langevin} with positions constrained to the
submanifold~\eqref{eq:manifold}, that is
\begin{equation}
  \label{eq:langevinconst}
  \left\{
  \begin{aligned}
     \mathrm{d}X_t & = Y_t \,\mathrm{d}t,\\
    \mathrm{d}Y_t & = -\nabla H_n(X_t)\,\mathrm{d}t -\gamma Y_t\, \mathrm{d}t
    + \sqrt{\frac{2\gamma}{\beta_n}}\,\mathrm{d}W_t
    + \nabla \xi(X_t)\,\mathrm{d}\theta_t,
    \\ \xi(X_t) & = 0.
  \end{aligned}
  \right.
\end{equation}
In~\eqref{eq:langevinconst}, $(\theta_t)_{t\geq 0}\in\dR^m$ is an explicit semi-martingale which
plays the role of a Lagrange multiplier enforcing the dynamics to stay
on~$\cM$ (see~\cite[Equation~(3.1)]{MR2945148} for the explicit expression,
which is however not useful for numerical simulations). Let us emphasize that
the position constraint induces a hidden constraint on the momenta
in~\eqref{eq:langevinconst}, which reads
\[
\forall\,t\geq 0,\quad \nabla \xi(X_t)^T Y_t =0.
\]
The above relation is obtained by taking the derivative of $t\mapsto\xi(X_t)$
along the dynamics~\eqref{eq:langevinconst}. This implies that momenta are
tangent to the submanifold's zero level set, which is a natural geometric
constraint~\cite{MR2945148,lelievre2018hybrid}. However, the
dynamics~\eqref{eq:langevinconst} does not sample from the conditioned
measure~\eqref{eq:constxi}, as shown in the following
proposition~\cite{MR2945148,lelievre2018hybrid}.
\begin{proposition}[Invariant measure]\label{prop:invmeaslangevin}
  The dynamics~\eqref{eq:langevinconst} has a unique invariant measure with
  marginal distribution in position given by
  \[
    \frac{\mathrm{e}^{- \beta_n H_n(x) } }{\overline{Z}_n^\xi}\sigma_{\cM}(\mathrm{d}x),
  \]
  where~$\overline{Z}_n^\xi$ is a normalization constant.
\end{proposition}
Although~\eqref{eq:langevinconst} does not sample from~$P_n^\xi$, we have seen
in Section~\ref{sec:dirac} how to fix this problem. More precisely,
Proposition~\ref{prop:coarea} shows that the dynamics~\eqref{eq:langevinconst}
run with the modified Hamiltonian~$H_n^\xi$ defined in~\eqref{eq:Hcorrect}
samples from~$P_n^\xi$.

However, in practice, it may be preferable not to use the gradient of~$U_n$,
since it involves the Hessian of the constraint~$\xi$ and may be cumbersome to
compute. Therefore, we will not run the dynamics~\eqref{eq:langevinconst} with
the modified Hamiltonian~$H_n^\xi$ but with~$H_n$, and perform some
reweighting to correct for the bias arising from the missing factor
$|\mathrm{det}\,G(x)|^{-\frac{1}{2}}$. As explained in Remark~\ref{rem:modif}
below in the context of a HMC discretization, this ensures that we are
sampling from the correct target distribution while only moderately increasing
the rejection rate.

\subsubsection{Discretization}
\label{sec:langevindiscretization}

In order to make a practical use of~\eqref{eq:langevinconst} combined with
Proposition~\ref{prop:coarea}, we need to define a discretization scheme. We
present below the strategy proposed by~\cite{lelievre2018hybrid}, which relies
on a second order discretization of~\eqref{eq:langevinconst} with a
Metropolis--Hastings selection and a reversibility check.

As discussed after~\eqref{eq:langevinconst}, momenta are tangent to the level
sets of the submanifold. We introduce
\begin{equation}
  \label{eq:projMp}
  \Pi_{\cM^\perp} = \mathrm{Id} - \nabla \xi(x) G^{-1}(x) \nabla\xi(x)^T \in (\dR^d)^n,
\end{equation}
whose action is to project the momentum orthogonally to the submanifold~$\cM$.
We next define the RATTLE scheme, which is a second order discretization of
the Hamiltonian part of~\eqref{eq:langevinconst}. 
\begin{algo}[RATTLE]
  \label{algo:RATTLE}
  Starting from a configuration $(x_m,y_m)$ with $x_m\in\cM$ and $\nabla \xi(x_m)^T y_m =0$,
  \begin{enumerate}
    \setlength\itemsep{0.6em}
  \item $\displaystyle y_{m+\frac{1}{4}} = y_m - \frac{\Delta t}{2} \nabla H_n(x_m)$;
  \item $\displaystyle x_{m + \frac{1}{2}} = x_m + \Delta t y_{m+\frac{1}{4}}$;
  \item compute the Lagrange multiplier~$\theta_m\in\dR^m$ associated
    with~$x_{m + \frac{1}{2}}$ to enforce the constraint, using
    Algorithm~\ref{algo:newton} below (if convergence has been reached);
  \item project as
    $\displaystyle x_{m+1} = x_{m + \frac{1}{2}} + \nabla\xi(x_m)\theta_m$
    and
    $\displaystyle y_{m + \frac{1}{2}} = y_{m+\frac{1}{4}} + \nabla
    \xi(x_m)\theta_m/\Delta t$;
  \item $\displaystyle y_{m+\frac{3}{4}} = y_{m+\frac{1}{2}} - \frac{\Delta t}{2} \nabla H_n(x_{m+1})$;
  \item $\displaystyle y_{m+1} = \Pi_{\cM^\perp} y_{m+\frac{3}{4}}$ where the projector $\Pi_{\cM^\perp}$ is
    defined in~\eqref{eq:projMp}.
  \end{enumerate}
  Finally, return $(x_{m+1},y_{m+1})$.
\end{algo}
A particular feature of this
numerical integrator is to be reversible up to momentum reversal (\emph{i.e.} evolving $(\hat{x}_{m+1},-\hat{y}_{m+1})$ with one step of RATTLE leads to $(\hat{x}_{m},-\hat{y}_{m})$), provided the same Lagrange multipliers are found when integrating from~$(\hat{x}_{m+1},-\hat{y}_{m+1})$. This property is crucial in order to simplify the
Metropolis--Hastings selection rule, see Step~(4) in Algorithm~\ref{algo:GHMC} below. This is true for sufficiently small timesteps, see~\cite[Section~VII.1.4]{hairer2006geometric} for further details. For larger timesteps, some care is needed since reversibility may be lost, as we discuss below.

%

We can now present the algorithm used to sample the conditioned distribution
by integrating~\eqref{eq:langevinconst}, which runs as follows. First, the
momenta~$y_m$ are updated to~$\tilde{y}_m$ according to the Ornstein--Uhlenbeck
process in~\eqref{eq:langevinconst} projected orthogonally to the submanifold
with~$\Pi_{\cM^\perp}$. Next, we evolve the configuration $(x_m,\tilde{y}_m)$
with a RATTLE step, leading to~$(\hat{x}_{m+1},\hat{y}_{m+1})$. However,
reversibility may be lost in the procedure for two reasons: either it is not
possible to perform one step of RATTLE starting
from~$(\hat{x}_{m+1},-\hat{y}_{m+1})$, or the image
of~$(\hat{x}_{m+1},-\hat{y}_{m+1})$ differs from $(x_m,-\tilde{y}_m)$;
see~\cite[Section~2.2.4]{lelievre2018hybrid} for a more precise discussion and
graphical illustrations of these issues. In both cases, the RATTLE move is
rejected, and the configuration is updated as~$(x_m,-\tilde{y}_m)$ (mind the
fact that momenta need to be reversed here when rejecting; see for
instance~\cite[Section~2.1.4.1]{MR2681239} for a discussion of this point).
Finally, a Metropolis--Hastings acceptance rule corrects for the time step
bias in the sampling. The full algorithm reads as
follows~\cite{lelievre2018hybrid}.
\begin{algo}[Constrained HMC with reversibility check]\label{algo:GHMC}
    Fix $T>0$, $\Delta t >0$, $\gamma >0$, $\Km\geq1$, $\Nit=\CEIL{T/\Delta t}$
    and choose an initial configuration $(x_0,y_0)$ with $x_0\in\cM$ and
    $\nabla \xi(x_0)^T y_0 =0$ (possibly obtained by projection). Set also
    thresholds~$\varepsilon_{\mathrm{rev}},\varepsilon_{\mathcal{N}}>0$, and
    define~$\eta_{\Delta t} = \mathrm{e}^{-\gamma \Delta t}$. For
    $m=0,\hdots,\Nit-1$, run the following steps:
    \begin{enumerate}
      \setlength\itemsep{0.5em}
    \item resample the momenta as
      \[
      \tilde{y}_{m} = \Pi_{\cM^\perp}\left(
      \eta_{\Delta t}y_m+\sqrt{\frac{1-\eta_{\Delta t}^2}
      {\beta_n}}G_m \right),
      \]
      where $G_m$ are independent $dn$-dimensional standard Gaussian random
      variables;
    \item perform one step of the RATTLE scheme (Algorithm~\ref{algo:RATTLE}) starting
        from $(x_m, \tilde{y}_{m})$, providing
        $(\hat{x}_{m+1},\hat{y}_{m+1})$ if the Newton algorithm with~$\Km$,
        $\varepsilon_{\mathcal{N}}$ has converged; otherwise
        set~$(x_{m+1},y_{m+1})=(x_m,-\tilde{y}_m)$ and increment~$m$;
      \item compute a RATTLE backward step from
        $(\hat{x}_{m+1},-\hat{y}_{m+1})$,
        providing~$(x_{m}^{\mathrm{rev}},y_{m}^{\mathrm{rev}})$ if the Newton
        algorithm with~$\Km$, $\varepsilon_{\mathcal{N}}$ has converged. If
        the Newton algorithm has not converged or if
        $|x_m-x_{m}^{\mathrm{rev}}|> \varepsilon_{\mathrm{rev}}$, reject the
        move by setting~$(x_{m+1},y_{m+1})=(x_m,-\tilde{y}_m)$ and
        increment~$m$;
    \item compute the Metropolis--Hastings ratio
      \begin{equation}
        \label{eq:MHratio}
        p_m = 1\wedge \mathrm{exp}\left[- \beta_n\left(
            H_n^\xi(\hat{x}_{m+1}) +\frac{|\hat{y}_{m+1}|^2}{2}
            -H_n^\xi(x_{m})  -\frac{|\tilde{y}_{m}|^2}{2}
          \right)\right],
      \end{equation}
      and set
      \[
        (x_{m+1},y_{m+1})=\left\{
          \begin{aligned}
            (\hat{x}_{m+1},\hat{y}_{m+1}) & \ \text{with probability}\ p_m,
            \\
            (x_m, - \tilde{y}_{m}) & \ \text{with probability}\ 1 - p_m.
          \end{aligned}
        \right.
      \]
    \end{enumerate}
  \end{algo}

  A particularity of our implementation with respect
  to~\cite{lelievre2018hybrid} is that we run the dynamics with the
  Hamiltonian~$H_n$ while the Metropolis--Hastings ratio~\eqref{eq:MHratio}
  (Step~(4) of Algorithm~\ref{algo:GHMC}) is computed with the modified
  Hamiltonian~$H_n^\xi$. As pointed out in Remark~\ref{rem:modif} below, the
  modification induced by the correction term~$U_n$ in~\eqref{eq:Hcorrect} is
  generally small. Therefore, considering~$H_n$ for the dynamics allows to
  avoid the computation of the Hessian of~$\xi$, while the selection rule
  corrects for this small error.

  In order for our description to be complete, we define how to project the
  position on~$\cM$ (Step~(3) in Algorithm~\ref{algo:RATTLE}). We use a
  variant of Newton's algorithm defined below.
  \begin{algo}[Newton algorithm]
    \label{algo:newton}
    Consider a tolerance threshold~$\varepsilon_{\mathcal{N}}>0$ and a maximal
    number of steps~$\Km\geq1$. Starting from an initial
    position~$x^0\notin\cM$ and a Lagrange multiplier $\theta^0 = 0\in\dR^m$,
    the projection procedure reads as follows: while $k\leq \Km$,
    \begin{enumerate}\setlength\itemsep{0.5em}
    \item compute $M_k = \nabla \xi(x^0)^T \nabla\xi (x^k)\in\dR^{m\times m}$;
    \item set $\theta^{k+1} = \theta^k - M_k^{-1}\xi(x^k)$;
    \item define the new position $x^{k+1} = x^k + \nabla \xi(x^0)^T \theta^{k+1}$;
    \item if $\mathrm{max}\big(|\theta^{k+1}-\theta^{k}|,|\xi(x^k)|\big)
      \leq \varepsilon_{\mathcal{N}}$,
      the algorithm has converged, else go back to Step~(1).
    \end{enumerate}
    If the algorithm has converged in~$k\leq \Km$ steps, return the
    value~$\theta^k$ of the Lagrange multiplier.
  \end{algo}
  We emphasize that a fixed direction~$\nabla\xi(x^0)$ is considered for
  projection, which is needed to preserve the reversibility property of the
  final algorithm~\cite{lelievre2018hybrid}. The procedure works provided the
  matrix~$M_k$ defined in Step~(1) is indeed invertible at each step of the
  inner loop, and we refer to~\cite{lelievre2018hybrid} for more details. We
  also mention that we can consider different stopping criteria for the Newton
  algorithm (Step~(4) in Algorithm~\ref{algo:newton}), for instance by using
  the relative error $|x^{k+1}-x^k|$. We are now ready to use
  Algorithm~\ref{algo:GHMC} to sample from the constrained distribution,
  challenge the theoretical results of Sections~\ref{sec:quadratic}
  and~\ref{sec:LDP} and explore conjectures. An example of implementation is
  provided in the arXiv version of this paper~\cite{CFS19}.

  \begin{remark}[Rejection sources]
    For a standard HMC scheme, rejection is only due to the
    Metropolis--Hastings selection (Step~(4) in Algorithm~\ref{algo:GHMC}).
    Here, rejection can be due to the following reasons:
    \begin{itemize}
    \item the Newton algorithm in Step~(2) (forward move) has not converged;
    \item the Newton algorithm in Step~(3) (backward or reversed move) has not
      converged;
    \item the reversibility check in Step~(3) has failed;
    \item the Metropolis rule in Step~(4) has rejected the step.
    \end{itemize}
    In any case, the first step resamples the momentum variable according to
    the Ornstein--Uhlenbeck process part in~\eqref{eq:langevin}, and rejection
    comes with a reversal of momenta. Let us also mention that, when the
    ratio~\eqref{eq:MHratio} is computed with~$H_n$ (up to an additive
    constant), the Metropolis rejection rate should scale as~$\Dt^3$. This
    will be the case in our situation when~$\xi$ and~$\varphi$ are linear
    since in this case the additional term~$U_n$ in~\eqref{eq:Hcorrect} is
    constant. This rate of decay is confirmed by numerical simulations (see
    Figure~\ref{fi:rate}).
  \end{remark}
  
  \begin{remark}[Correction term]
    \label{rem:modif}
    Proposition~\ref{prop:coarea} shows that the Hamiltonian of the system
    must be modified in order for the constrained
    dynamics~\eqref{eq:langevinconst} to sample from the probability
    distribution~\eqref{eq:Pnxi}. However, in Algorithm~\ref{algo:GHMC}, we
    run the dynamics with~$H_n$ and perform the selection with~$H_n^\xi$. This
    is motivated by the following scaling argument. Consider
    \[
    \xi(x) = \frac{1}{n}\sum_{i=1}^n \varphi(x_i)
    \]
    for some real-valued smooth function~$\varphi$, which corresponds to the
    linear constraint situation described in Section~\ref{sec:linear}. In this
    case, the corrector term in~\eqref{eq:Hcorrect} reads
    \[
      U_n(x) = -\frac{1}{2 \beta_n} \log \left(\sum_{i=1}^n |\nabla
        \varphi(x_i)|\right),
    \]
    up to an additive constant. This means that the correction term
    in~\eqref{eq:Hcorrect} scales like~$\mathrm{O}(\log(n)/n^2)$ when
    $\beta_n = \beta n^2$, whereas the remainder of the Hamiltonian
    is~$\mathrm{O}(1)$. As a result, the correction is much smaller than the
    Hamiltonian energy~$H_n$, and we may neglect it in the dynamics. This
    allows to avoid computing the Hessian of the constraint~$\xi$ at the price
    of a small increase in the rejection rate.
  \end{remark}

\subsection{Numerical results}
\label{sec:applications}

\subsubsection{Linear statistics with linear constraint: the influence of confinement}\label{sec:applin}

Since one motivation for our work was to study the trace constraint with
quadratic confinement, as detailed in Section~\ref{sec:quadratic}, we first
consider the model presented in Theorem~\ref{th:Hermite} with $d=2$,
$\beta_n =n^2$ and
\begin{equation}
  \label{eq:philinsimu}
\forall\,x\in\dR^2,\quad \varphi(x) = c - x\cdot v ,\quad v= 
\begin{pmatrix}
  1 \\ 0
\end{pmatrix}.
\end{equation}
We run Algorithm~\ref{algo:GHMC} setting $n=300$, $T=10^6$, $\Delta t = 0.5$,
$\gamma =1$ and
$\varepsilon_{\mathcal{N}} = \varepsilon_{\mathrm{rev}} = 10^{-12}$ with
$\Km=20$. In all the simulations in dimension~2, the initial configuration is
drawn uniformly over~$[-1,1]^2$. We first set $V(x)=|x|^2$ and $c=1$, so that
according to Theorem~\ref{th:Hermite}, the conditional law of the empirical
measure~$\mu_n$ under~$P_n$ with the constraint $\mu_n(\varphi)=0$ should
converge in the limit of large~$n$ towards a unit disk centered at $(1,0)$
in~$\dR^2$. The simulations presented in Figure~\ref{fi:quadlin} show a very
good agreement with the expected result.

\begin{figure}[!ht]
  \centering
  \includegraphics[width=.49\textwidth]{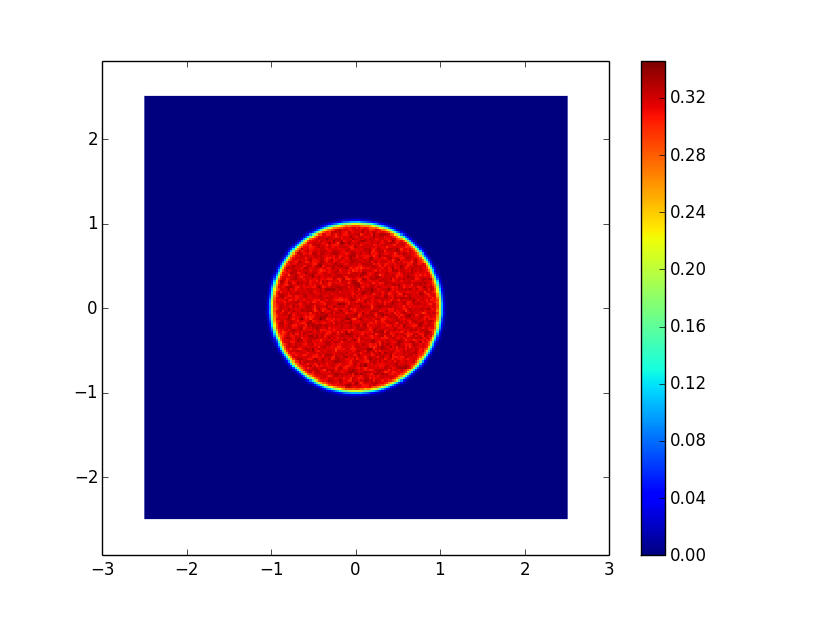}
  \includegraphics[width=.49\textwidth]{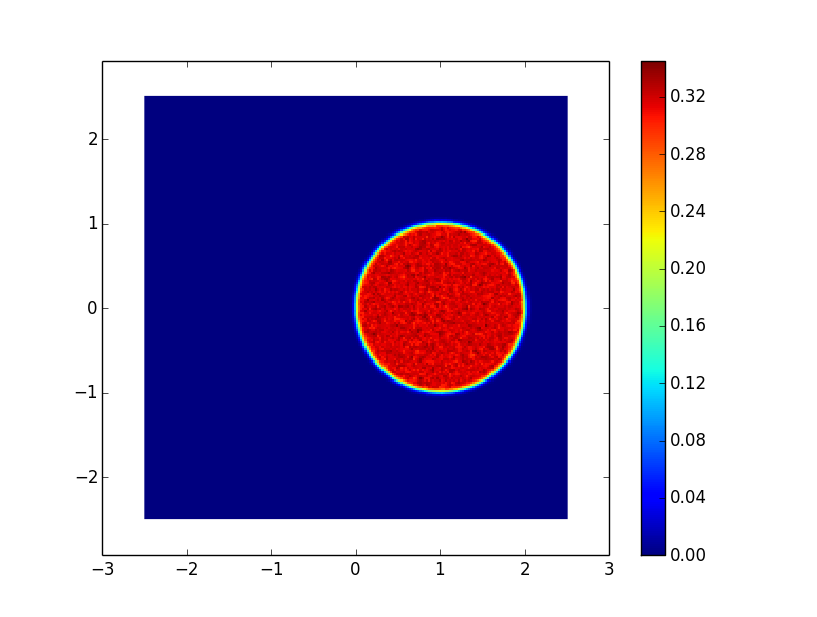}  
  \caption{\label{fi:quadlin} Study of the quadratic confinement for $n=300$
    without constraint (left) and with the constraint~\eqref{eq:philinsimu}
    (right). We see that the constrained measure is a unit disk 
    centered at~$(1,0)$.}
\end{figure} 

In this simple case, the Hamiltonian in~\eqref{eq:Hcorrect} is only modified
by a constant, so we expect the Metropolis rejection rate (Step~(5) in
Algorithm~\ref{algo:GHMC}) to scale like~$\mathrm{O}(\Delta t^3)$ when
$\Delta t\to 0$. In Figure~\ref{fi:rate}, we plot this rate in log-log
coordinates (setting here $n=50$ to reduce the computation time). The slope is
indeed close to~$3$, which confirms our expectation.

\begin{figure}[!ht]
  \centering
  \includegraphics[width=.6\textwidth]{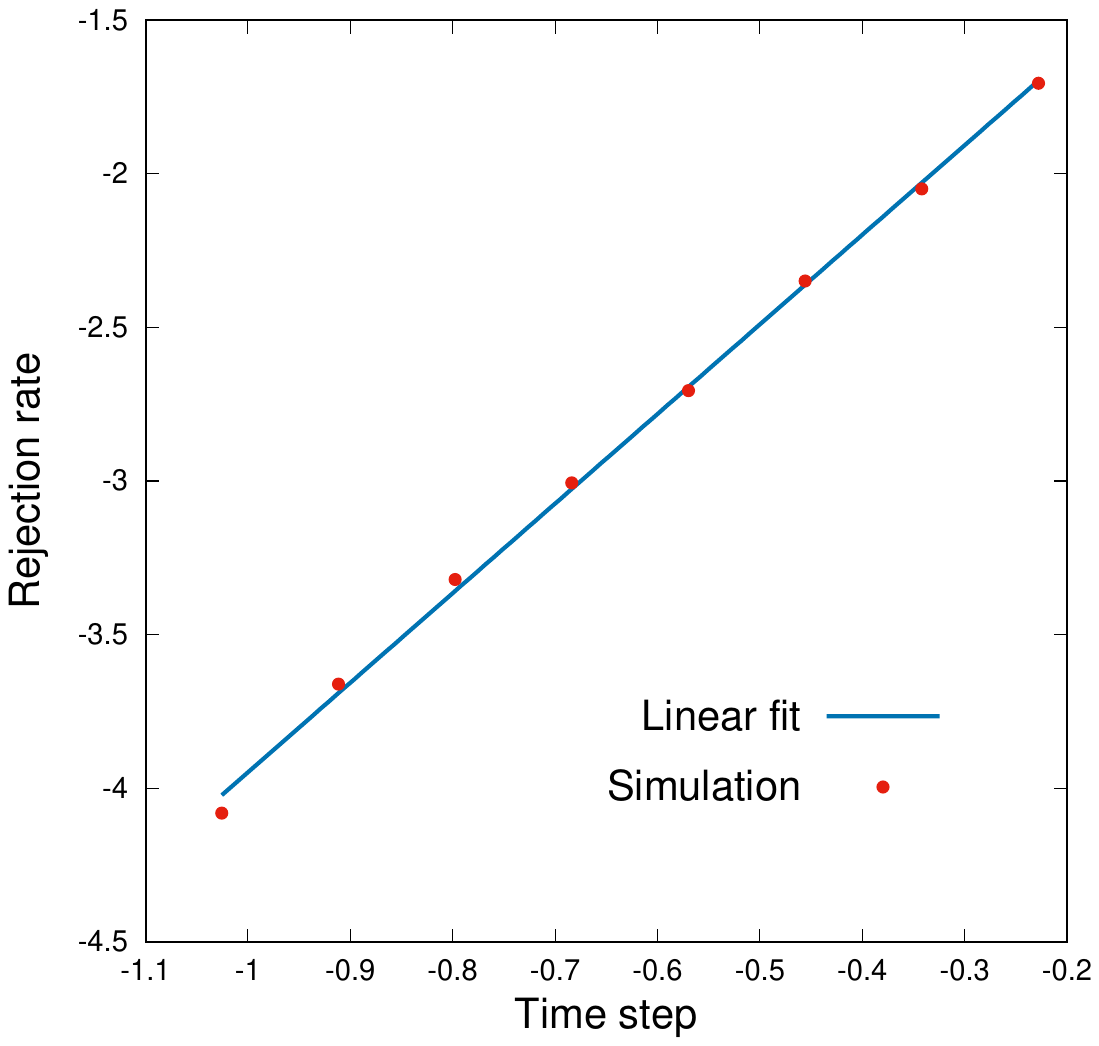}
  \caption{\label{fi:rate} Study of the rejection rate of the
    Metropolis--Hastings selection rule with $n=50$ (Step~(5) in
    Algorithm~\ref{algo:GHMC}) in log-log coordinates. The slope of the linear
    fit is about~$2.9$.}
\end{figure} 

In order to show that the translation phenomenon is specific to the quadratic
confinement, we first consider the case of a quartic confinement potential,
namely $V(x) = |x|^4/4$ subject to the constraint~\eqref{eq:philinsimu} with
$c=0.5$. This choice for~$V$ together with~$\varphi$ defined
in~\eqref{eq:philinsimu} satisfies Assumption~\ref{as:intB}, so that
Theorem~\ref{th:lincond} applies. However, no analytic solution is a priori
available because the rotational symmetry is lost. The unconstrained
equilibrium measure in Figure~\ref{fi:quarticlin} (left) shows a depletion of
the density around~$(0,0)$. In Figure~\ref{fi:quarticlin} (right), we observe
that the shape of the distribution is significantly modified by the
constraint, and does not possess any rotational invariance. As could have been
expected, the particles close to the origin feel a weaker confinement, so the
distribution is more concentrated near the outer edge.

\begin{figure}[!ht]
  \centering
  \includegraphics[width=.49\textwidth]{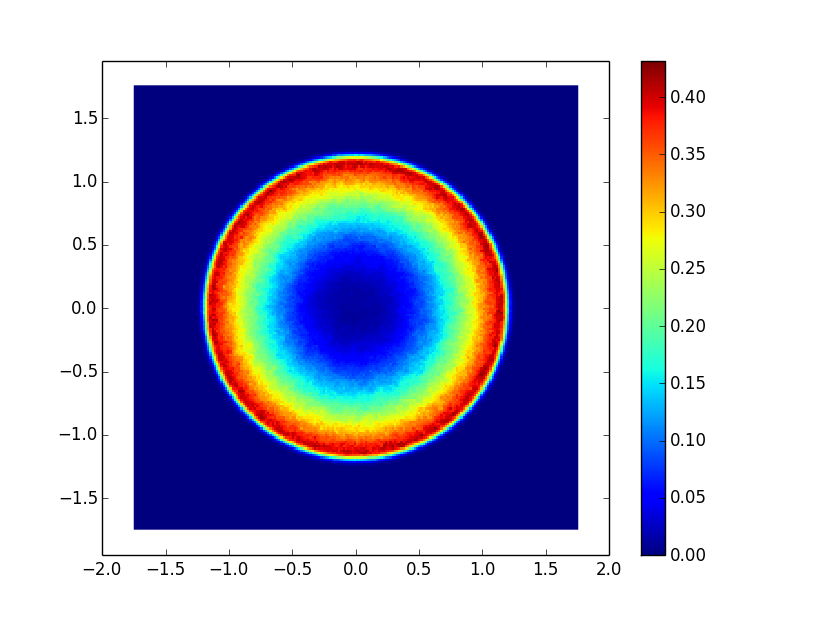}
  \includegraphics[width=.49\textwidth]{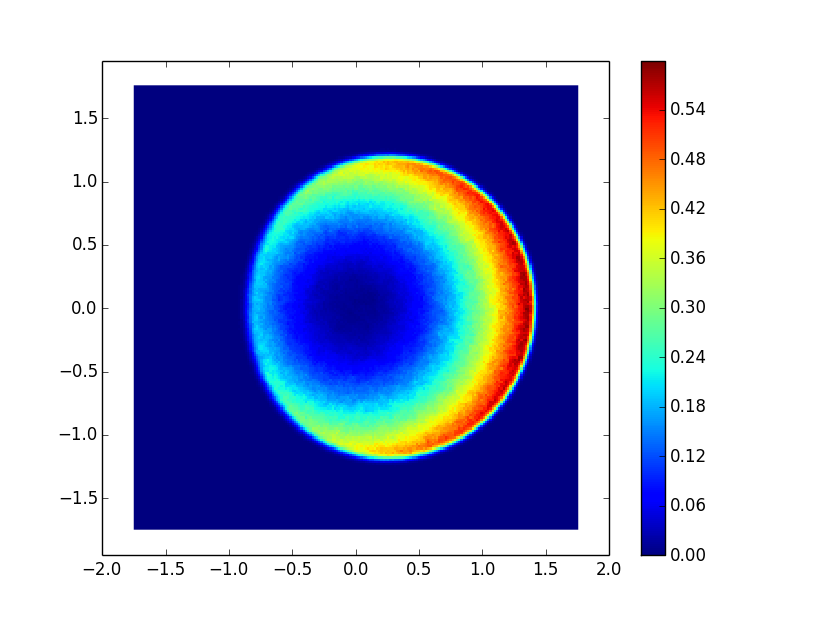}  
  \caption{\label{fi:quarticlin} Study of the quartic confinement for $n=300$
    without constraint (left) and with the constraint~\eqref{eq:philinsimu}
    (right). The shape of the equilibrium measure is significantly
    distorted by the constraint.}
\end{figure} 

Another interesting case is when the confinement is weaker than quadratic,
\emph{e.g.} $V(x) = \frac{2}{3}|x|^{\frac{3}{2}}$, for which
Theorem~\ref{th:lincond} still applies. The results are shown in
Figure~\ref{fi:weaklin}, considering again the
constraint~\eqref{eq:philinsimu} with $c=0.5$. We observe that the shape of
the distribution also significantly changes by spreading in the direction of
the constraint. This can be interpreted as follows: since the confinement is
stronger at the origin, the more likely way to observe a fluctuation of the
barycenter (or less costly in terms of energy) is in this case to spread the
distribution.

Quite interestingly, for both potentials the distribution obtained
as~$c\to +\infty$ seems to reach a limiting ellipsoidal shape, under an
appropriate rescaling (figures not shown here). Studying more precisely these
limiting shapes and the rate at which they appear is an interesting open
problem.

\begin{figure}[!ht]
  \centering
  \includegraphics[width=.49\textwidth]{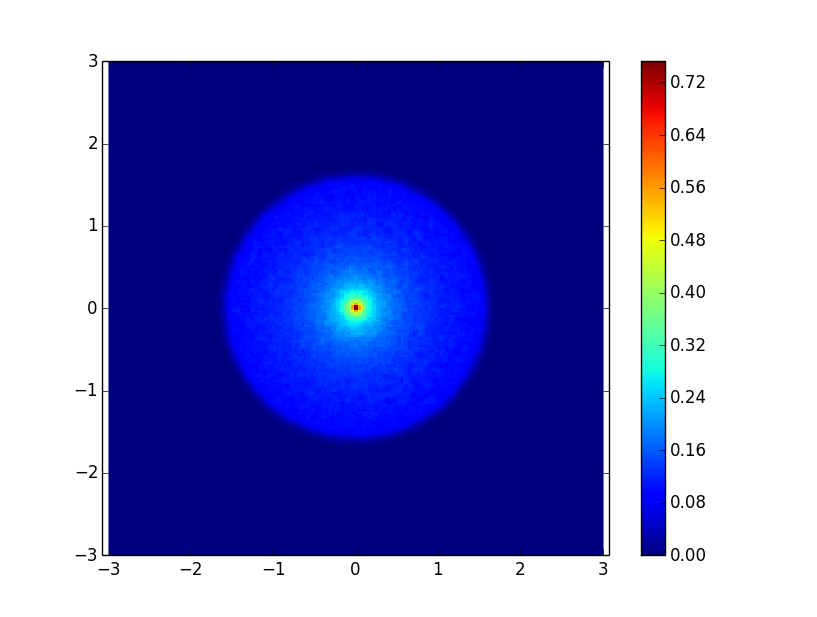}
  \includegraphics[width=.49\textwidth]{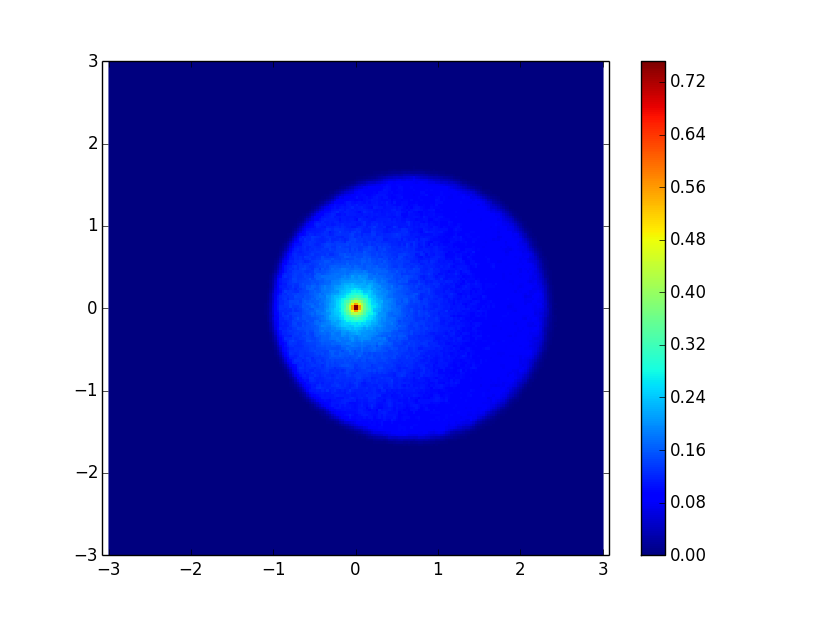}  
  \caption{\label{fi:weaklin} Study of the weak confinement for $n=300$ without
    constraint (left) and with the constraint~\eqref{eq:philinsimu} (right).
    The constraint now spreads the equilibrium measure to the right.
  }
\end{figure}

\subsubsection{Other constraints in dimension two}

In order to illustrate the efficiency of our algorithm in situations richer
than the linear constraint with a linear function~$\varphi$, we now present
two other cases. First, we keep a linear constraint with $V(x)=|x|^2$, but set
\[
  \forall\,x\in\dR^2,\quad \varphi(x) = c - \frac{\cos(5x_1) +\cos(5x_2)}{2},
\]
where~$x_1$ and~$x_2$ denote here the first and second coordinates of
$x\in\dR^2$. This choice is motivated by Remark~\ref{rm:solutionlinear}: since
the Laplacian of~$\varphi$ takes positive and negative values, we expect the
particles to concentrate in some regions of~$\dR^2$, possibly leading to a
phase separation. Note also that, in order for the two last conditions in
Assumption~\ref{as:intB} to be satisfied, we need to choose $c\in(-1,1)$. We
set again $n=300$ but $\Dt = 0.4$ to reduce the rejection rate. The other
parameters are the same as in Section~\ref{sec:applin}. We plot in
Figure~\ref{fi:lincos} the result of the simulation for~$c=0.2$ and $c=0.5$.
The particles concentrate in the regions where the cosines are higher, which
seems to lead to a phase separation when~$c$ comes close to~$1$.

\begin{figure}[!ht]
  \centering
  \includegraphics[width=.49\textwidth]{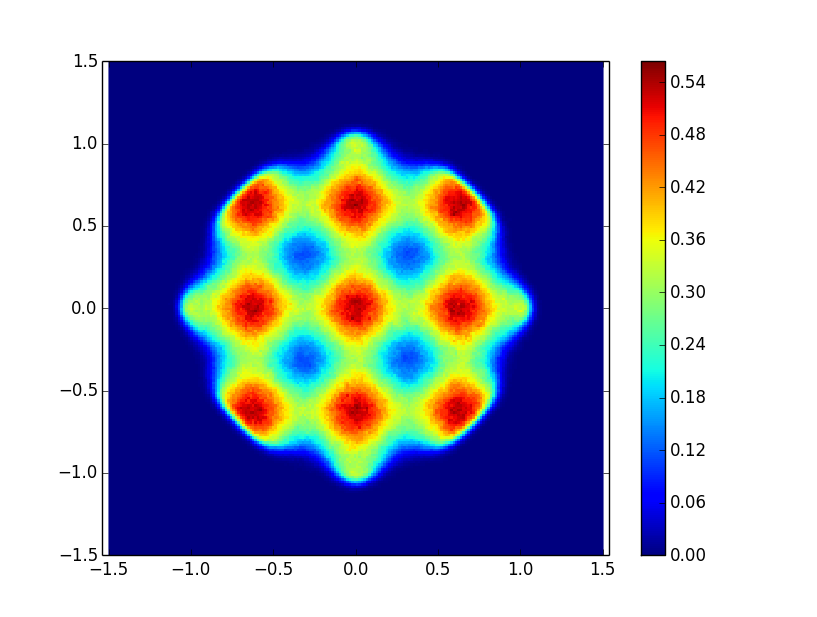}
  \includegraphics[width=.49\textwidth]{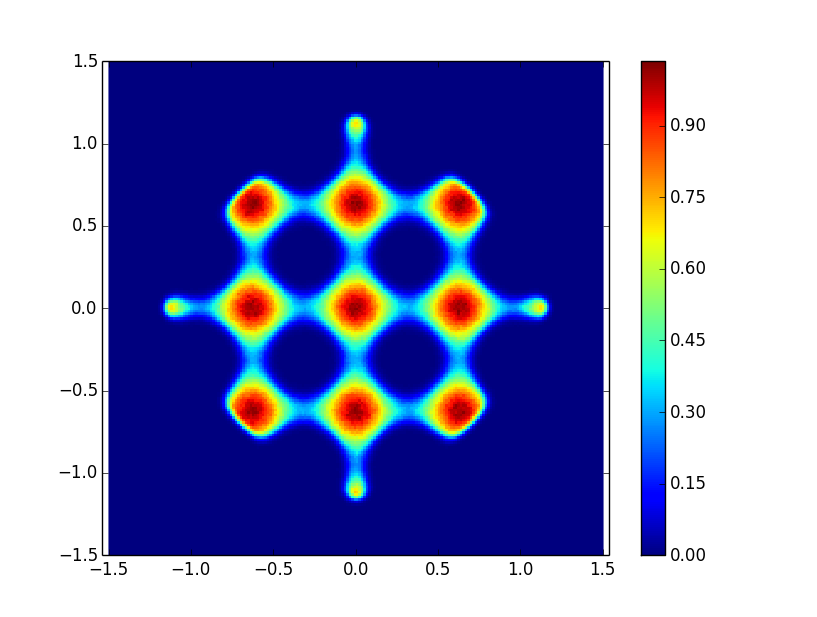}  
  \caption{\label{fi:lincos} Study of the cosine constraint for $n=300$ with
    $c=0.2$ (left) and $c=0.5$ (right). A phase separation appears as the particles are
    constrained to stay around the local maxima of the cosines.}
\end{figure} 

In order to illustrate the results of Section~\ref{sec:quadstats}, we consider
a quadratic constraint $\xi:(\dR^d)^n\to\dR$ of the form
\begin{equation}
  \label{eq:simuquadcond}
\forall\,x\in(\dR^d)^n,\quad \xi(x)=  \frac{1}{n^2}\sum_{i,j=1}^n \psi(x_i, x_j),
\end{equation}
with, for $x,y\in\dR^2$,
\begin{equation}
  \label{eq:simuquadphi}
\psi(x,y)=\phi(x-y),
\quad \mbox{and}\quad
\phi(x) = c - |x|.
\end{equation}
A motivation for this choice is to modify the \emph{rigidity} of the gas by
constraining the particles to be closer or farther apart one from another in
average. In order to make this rigidity anisotropic, we also
consider~\eqref{eq:simuquadphi} with
\begin{equation}
  \label{eq:quadanisotrope}
  \forall\,x\in\dR^2,
  \quad
  \phi(x) = c - |x_1|.
\end{equation}
The choice~\eqref{eq:quadanisotrope} modifies the rigidity only in one
direction. For illustration we take $V(x)=|x|^4/4$, $n=50$, $\Dt = 0.5$,
$T=10^6$ (we take a lower number of particles because the constraint makes the
dynamics quite stiff). We set $c=1$ for~\eqref{eq:simuquadphi} and $c=0.5$
for~\eqref{eq:quadanisotrope}, which forces the particles to move away from
each other. These choices for~$\psi$ satisfy the conditions of
Proposition~\ref{prop:suffquad}, and the application~$Q$ defined
in~\eqref{eq:Qdef} can be proved to be convex, so Assumption~\ref{as:psi} is
satisfied and Theorem~\ref{th:quadcond} applies. The distribution obtained for
the constraint~\eqref{eq:simuquadphi}, presented in Figure~\ref{fi:quadstat}
(left), shows that the more likely way for the particles to be repelled by the
constraint induced by~$\psi$ is to move away from the center and concentrate
on the edge, compared to Figure~\ref{fi:quarticlin} (left). For the
constraint~\eqref{eq:quadanisotrope}, we clearly observe in
Figure~\ref{fi:quadstat} (right) the effect of anisotropy.

\begin{figure}[!ht]
  \centering
  \includegraphics[width=.49\textwidth]{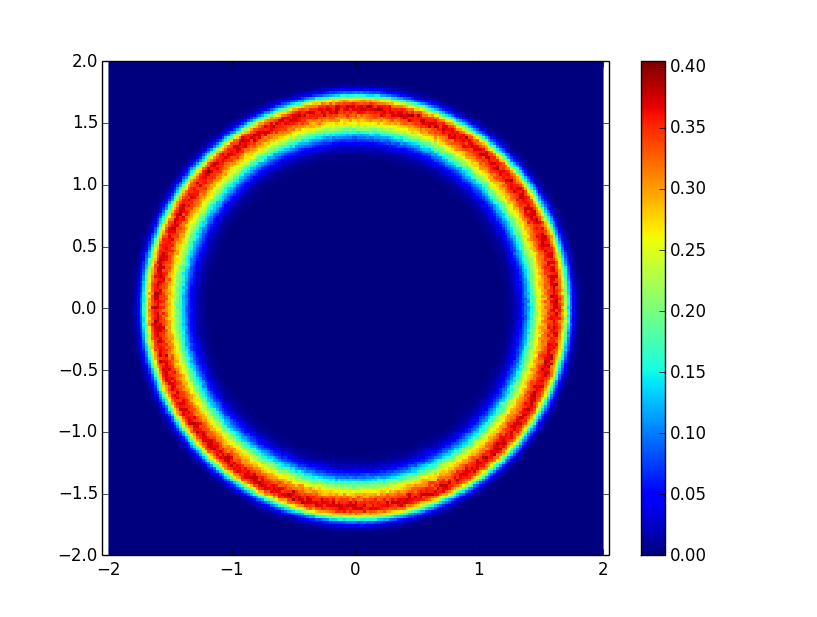}
  \includegraphics[width=.49\textwidth]{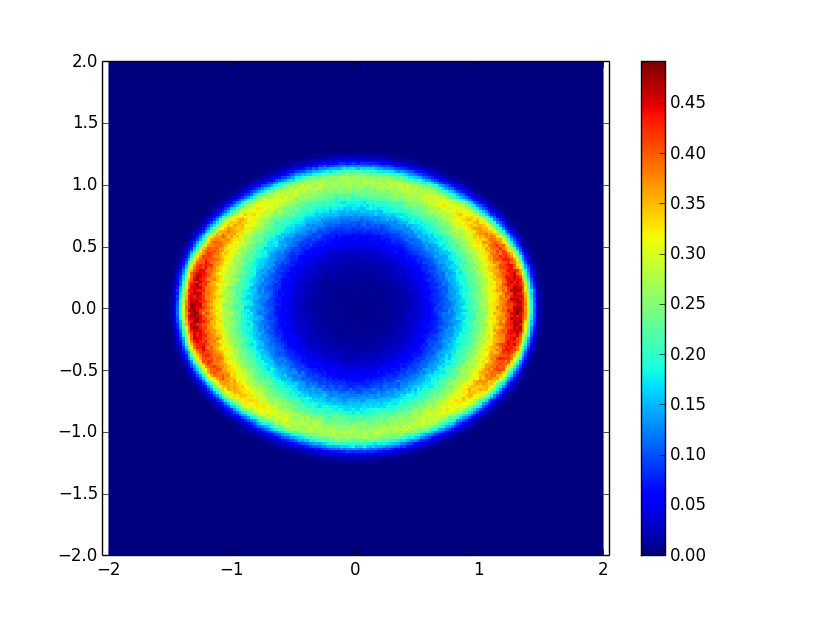}
  \caption{\label{fi:quadstat} Study of the quartic confinement for $n=50$
    with the quadratic statistics
    constraint~\eqref{eq:simuquadcond}-\eqref{eq:simuquadphi} where $c=1$
    (left), and with the constraint~\eqref{eq:quadanisotrope} with $c=0.5$
    (right). This has to be compared to the unconstrained distribution in
    Figure~\ref{fi:quarticlin} (left).}
\end{figure} 

\subsubsection{A one dimensional example}

We consider the Gaussian Unitary Ensemble (GUE), which is a degenerate
two-dimensional Coulomb gas for which the particles are confined on the real
axis. It corresponds in a sense to \eqref{eq:Pn} with $d=1$, $V(x)=|x|^2$ but
$g(x)=-\log|x|$, and $\beta_n=n^2$. It is known that the equilibrium measure
is then the Wigner semi--circle law, and we refer for instance to
\cite{arous1997large} for a large deviations study. We can apply
Theorem~\ref{th:Hermite} for the linear constraint~\eqref{eq:philinsimu}. In
this case, the Wigner semi--circle law is indeed translated by a factor~$c$
(figure not shown here).

Next, in order to illustrate a case which is not covered by our analysis, we
want to sample the spectrum of those matrices whose determinant is equal to~$\pm 1$. In our
context, this corresponds to the configurations~$x\in(\dR^d)^n$ with
$\prod_{i=1}^n |x_i| =1$. By taking the logarithm, this constraint is actually
of the form~\eqref{eq:constaverage} (by Remark~\ref{rem:parametrization} the
conditioned probability measure~\eqref{eq:constxi} does not depend on the parametrization) with
\begin{equation}
  \label{eq:xilog}
  \forall\, x\in(\dR^d)^n, \quad
  \xi(x) =\frac{1}{n} \sum_{i=1}^n \log|x_i|-c.
\end{equation}
We plot in Figure~\ref{fi:wignerlog} the distribution for $n=300$, $T= 10^5$
and $\Dt = 0.05$ for the unconstrained log-gas, and for the
constraint~\eqref{eq:xilog} with $c=-0.5$ and $c=0$ for~$\Dt=0.01$
(starting with particles equally spaced over the
interval~$[-1,1]$). We observe what looks like a
symmetrized Marchenko--Pastur distribution. Actually
Remark~\ref{rm:solutionlinear} suggests that the effective potential of the
constrained distribution is $|\cdot|^2- \alpha \log|\cdot|$ for some
$\alpha> 0$, which is not that far from the Laguerre potential 
$|\cdot|-\alpha\log|\cdot|$.

\begin{figure}[!ht]
  \centering
  \includegraphics[width=1.\textwidth,trim = 0cm 5cm 0cm 6cm, clip]{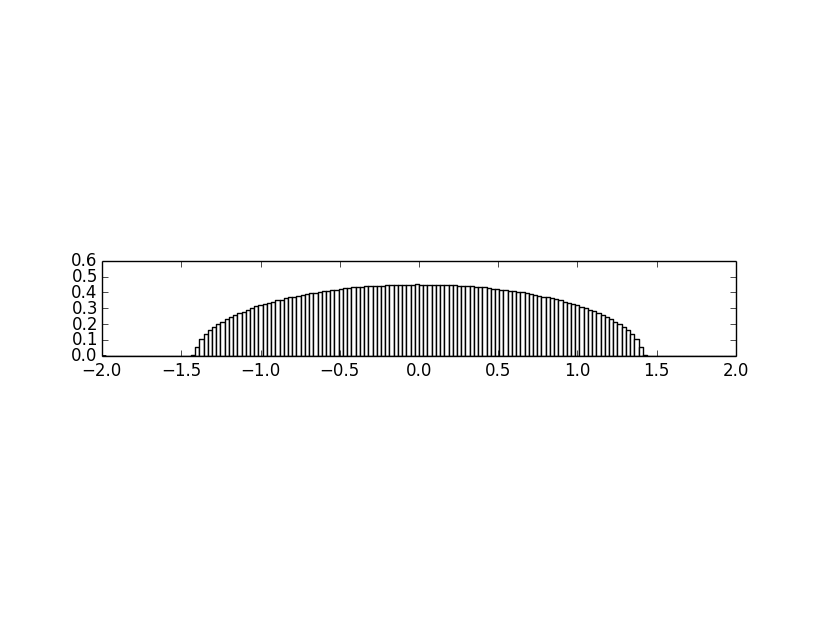}
  \includegraphics[width=1.\textwidth,trim = 0cm 5cm 0cm 6cm, clip]{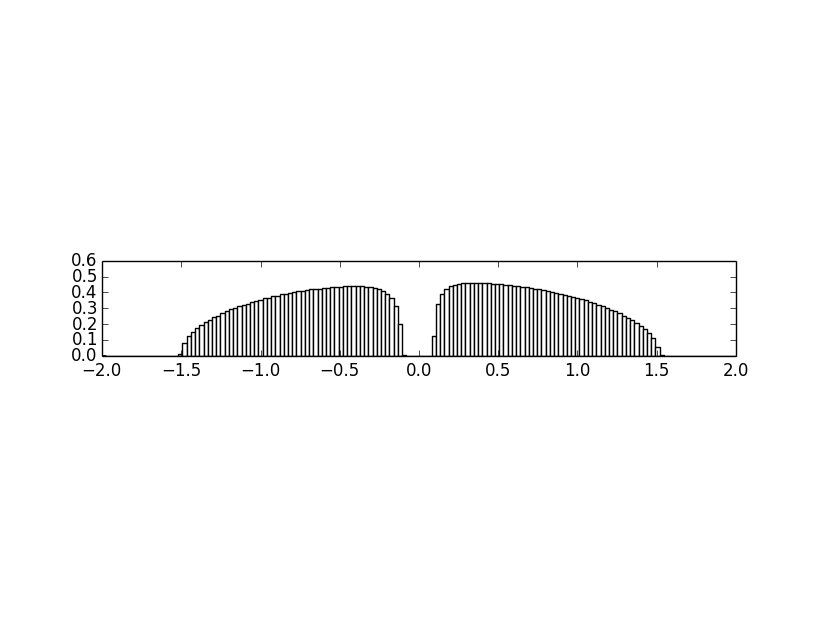}
  \includegraphics[width=1.\textwidth,trim = 0cm 5cm 0cm 6cm, clip]{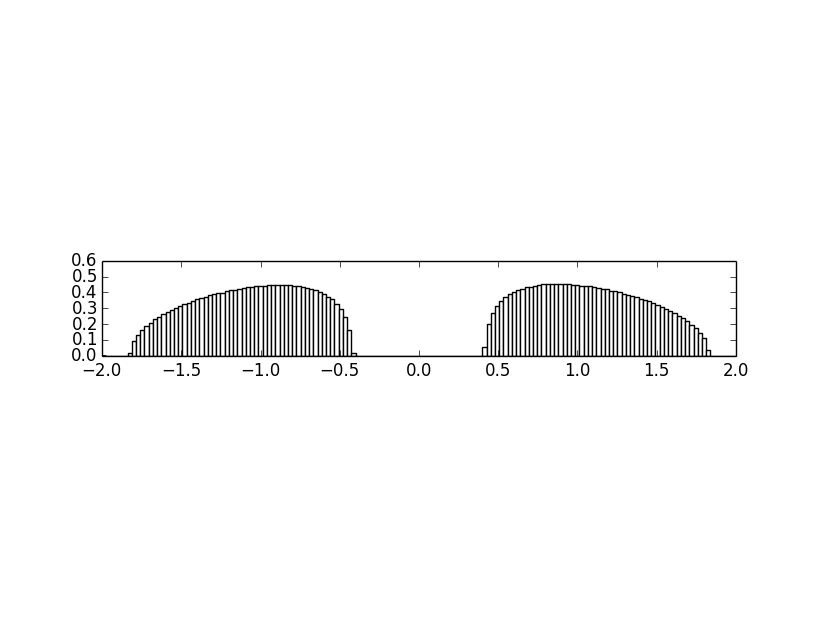}  
  \caption{\label{fi:wignerlog} Study of the one dimensional log-gas for
    $n=300$ without constraint (top) and with the
    constraint~\eqref{eq:philinsimu} for $c=-0.5$ (middle) and $c=0$ (bottom).
    This corresponds to a deformation of the semi--circle distribution.}
\end{figure}

\section*{Acknowledgements}

The authors warmfully thank the anonymous referee for his careful re-reading
of the manuscript and his insightful comments.
The PhD of Grégoire Ferré is supported by the Labex Bézout ANR-10-LABX-58-01.
The work of Gabriel Stoltz was funded in part by the Agence Nationale de la
Recherche, under grant ANR-14-CE23-0012 (COSMOS). Gabriel Stoltz is supported
by the European Research Council under the European Union’s Seventh Framework
Programme (FP/2007-2013)/ERC Grant Agreement number 614492, and also benefited
from the scientific environment of the Laboratoire International Associé
between the Centre National de la Recherche Scientifique and the University of
Illinois at Urbana-Champaign. 

\appendix
\section{Proof of Theorem \ref{th:Hermite}}
\label{sec:proofquad}

This section is devoted to the proof of Theorem~\ref{th:Hermite}. First, the
following lemma is some sort of quantitative Wasserstein version of
\cite[Lemma C.1]{MR2892961}.

\begin{lemma}[Translation]\label{le:w}
  Let $d\geq1$, $p\geq1$, $\mu_1,\mu_2\in\mathcal{P}_p(\mathbb{R}^d)$ and
  $\varphi:\mathbb{R}^d\to\mathbb{R}$ be a measurable function. Then for all
  $m_1,m_2\in\mathbb{R}^d$,
  \[
    \DWP^p(\mu_1*\delta_{m_1},\mu_2*\delta_{m_2})
    \leq 2^{p-1}|m_1-m_2|^p+2^{p-1}\mathrm{d}^p_{\mathrm{W}_p}(\mu_1,\mu_2).
  \]
  Moreover, if $m_i= a \left(\int\varphi\,\mathrm{d}\mu_i\right)$ for $i=1,2$
   and $a\in\mathbb{R}^d$, then
  \[
    \DWP(\mu_1*\delta_{m_1},\mu_2*\delta_{m_2})
    \leq 2^{\frac{p-1}{p}}(1+|a|^p\|\varphi\|^p_{\mathrm{Lip}})^{\frac{1}{p}}
    \DWP(\mu_1,\mu_2).
  \]
\end{lemma}

Note that the right hand side is infinite if~$\varphi$ is not Lipschitz.

\begin{proof}
  We have, by using the infimum formulation~\eqref{eq:inf_formulation_Wasserstein} of the distance
  $\mathrm{d}_{\mathrm{W}_p}$, 
  \[
    \DWP^p(\mu_1*\delta_{m_1},\mu_2*\delta_{m_2})
    \leq 2^{p-1}|m_1-m_2|^p+2^{p-1}\DWP^p(\mu_1,\mu_2),
  \]
  where we used the convexity inequality $|u+v|^p\leq 2^{p-1}(|u|^p+|v|^p)$
  valid for all $u,v\in\mathbb{R}^d$. Then,
  since $p\mapsto \DWP$ is monotonic for $p\geq 1$, it holds
  \[
  |m_1-m_2|
  =\left|a\left(\int_{\dR^d}\varphi\,\mathrm{d}(\mu_1-\mu_2)\right)\right|
  \leq|a|\|\varphi\|_{\mathrm{Lip}}\DW(\mu_1,\mu_2)
  \leq|a|\|\varphi\|_{\mathrm{Lip}}\DWP(\mu_1,\mu_2),
  \]
  which is the claimed estimate.
\end{proof}

The following lemma is a $d$-dimensional version of the factorization lemma
in~\cite{chafai-lehec}. It expresses a non obvious independence between the
center of mass and the shape of the cloud of particles distributed according
to~$P_n$. As noticed in~\cite{chafai-lehec}, it reminds the structure of
certain continuous spins systems such as in~\cite{MR2028218,MR1847094}.


\begin{lemma}[Factorization]\label{le:factor}
  Suppose that the assumptions of Theorem~\ref{th:Hermite} are satisfied, and
  define $u =(v,\hdots,v)\in(\dR^d)^n$. Let~$\pi$ and~$\pi^\perp$ be the
  orthogonal projections in~$(\mathbb{R}^d)^n$ on the linear subspaces
  \[
    L=\dR u
    \quad\text{and}\quad
    L^\perp = \left\{ x \in (\dR^d)^n:x\cdot u = 0\right\}.
  \]
  Then, abridging $X_n$ into $X$, the following properties hold:
  \begin{itemize}
  \item for all $x\in(\mathbb{R}^d)^n$,
    denoting $s(x)=\frac{x_1+\cdots+x_n}{n}\in\mathbb{R}^d$, we have
    \begin{align*}
      \pi(x)
      & = (s(x)\cdot v)u =\big((s(x)\cdot v)v,\ldots,(s(x)\cdot v)v\big), \\
      \pi^\perp(x)
      &=x-\pi(x)= x - (s(x)\cdot v)u=\big(x_1-(s(x)\cdot v) v,\ldots,x_n-(s(x)\cdot v)v\big);
    \end{align*}
  \item $\pi(X)$ and $\pi^\perp(X)$ are independent random vectors;
  \item $\pi(X)$ is Gaussian with law
    $\mathcal{N}\left(0,\frac{n}{2\beta_n }\right)u$, so that~$s(X)\cdot v$ has law
    $\mathcal{N}\left(0,\frac{n}{2\beta_n }\right)$;
  \item $\pi^\perp(X)$ has law of density proportional to
    $x\in L^\perp\mapsto\mathrm{e}^{-\beta_n H_n(x)}$ with respect to the
    trace of the Lebesgue measure on the linear subspace~$L^\perp$
    of~$\mathbb{R}^{dn-1}$.
  \end{itemize}
\end{lemma}

\begin{proof}[Proof of Lemma~\ref{le:factor}]
  Since $|v|=1$, we have $|u| = \sqrt{n}$, so the orthonormal projection
  on~$L$ reads, for $x \in (\dR^d)^n$,
  \[
  \pi(x) = \frac{x\cdot u}{|u|^2}u = \left(\frac{1}{n}\sum_{i=1}^n x_i\cdot v\right) u
  = \big( s(x)\cdot v\big)u.
  \]
  The expression of $\pi^\perp$ follows easily.
  For all $x\in(\mathbb{R}^d)^n$, from $x=\pi(x)+\pi^\perp(x)$ we get
  \[
    |x|^2=|\pi(x)|^2+|\pi^\perp(x)|^2.
  \]
  On the other hand, for all $i,j\in\{1,\ldots,n\}$ it holds
  \begin{align*}
    x_i-x_j
    &=\pi(x)_i+\pi^\perp(x)_i-\pi(x)_j-\pi^\perp(x)_j\\
    &=s(x)+\pi^\perp(x)_i-s(x)-\pi^\perp(x)_j\\
    &=\pi^\perp(x)_i-\pi^\perp(x)_j.
  \end{align*}
  Since $V(x)=|x|^2$, it follows that, for all
  $x=(x_1,\ldots,x_n)\in(\mathbb{R}^d)^n$,
  \[
    H_n(x)=
    \frac{1}{n}|x|^2
    +\frac{1}{n^2}\sum_{i\neq j}g(x_i-x_j)
    =\frac{1}{n}|\pi(x)|^2+H_n\big(\pi^\perp(x)\big).
  \]
  Now, let $u_1,\ldots,u_{dn}$ be an orthogonal basis of
  $(\mathbb{R}^d)^n=\mathbb{R}^{dn}$ with $u_1 = u/\sqrt{n}\in L$. For
  all $x\in(\mathbb{R}^d)^n$ we write
  $x=\sum_{i=1}^{dn}t_i(x)u_i$. We have $\pi(x)=t_1(x)u_1 = (s(x)\cdot v)u$ and
  $\pi^\perp(x)=\sum_{i=2}^{dn}t_i(x)u_i$. Then we have,
  for all bounded measurable $f:L\to\mathbb{R}$ and $g:L^\perp\to\mathbb{R}$,
  \begin{align*}
    \mathbb{E}\Big[f\big(\pi(X)\big)g\big(\pi^\perp(X)\big) \Big]
    &=Z^{-1}\int_{(\mathbb{R}^d)^n}f\big(\pi(x)\big)g\big(\pi^\perp(x)\big)\,\mathrm{e}^{-\frac{
        \beta_n}{n}|\pi(x)|^2}\,\mathrm{e}^{-\beta_n H_n(\pi^\perp(x))}\mathrm{d}x_1
    \cdots\mathrm{d}x_n\\
    &=Z^{-1}\left(\int_{\mathbb{R}}f(t')\,\mathrm{e}^{-\frac{\beta_n}{n} |t'|^2}\mathrm{d}t'\right)
      \left(\int_{\mathbb{R}^{dn-1}}g(t'')\,
      \mathrm{e}^{-\beta_n H_n(t'')}\mathrm{d}t''\right),      
  \end{align*}
  where $t'=t_1u_1$, $\mathrm{d}t'=\mathrm{d}t_1$,
  $t''=\sum_{i=2}^{dn}t_iu_i$
  and $\mathrm{d}t''=\prod_{i=2}^{dn}\mathrm{d}t_i$. This concludes the proof of
  the last two points of the lemma.
\end{proof}

We can now turn to the proof of Theorem~\ref{th:Hermite}.

\begin{proof}[Proof of Theorem~\ref{th:Hermite}]
  Thanks to Lemma \ref{le:factor}, we have, denoting again by $u=(v,\ldots,v)\in(\mathbb{R}^d)^n$,
  \begin{align}
    \mathrm{Law}\left(X_n\, \left|\, \frac{1}{n}\sum_{i=1}^n\varphi(X_i)=0\right. \right)
    & =\mathrm{Law}\left(X_n\, \left|\, \frac{1}{n}\sum_{i=1}^n X_{n,i}\cdot v= c\right.
    \right)\nonumber\\
    & =\mathrm{Law}\left(X_n\, \Big|\, s( X_{n})\cdot v= c\right)\nonumber\\
    & =\mathrm{Law}\left(X_n\, \Big|\, \big(s( X_{n})\cdot v\big)u= cu\right)\nonumber\\
    &=\mathrm{Law}\left(X_n\, \Big|\, \pi(X_n)= cu\right)\nonumber\\
    &=\mathrm{Law}\left(cu+\pi^\perp(X_n)\right)\nonumber\\
    &=\mathrm{Law}\big(\widetilde X_n\big),\label{eq:le:xtilde}
  \end{align}
  where
  $\widetilde{X}_n = cu + \pi^\perp(X_n)= cu + X_n - \pi(X_n)$. We also have
  \[
    \widetilde{X}_n
    =\Big( \big(c-s(X_n)\cdot v \big) v + X_{n,1},\ldots,
    \big(c-s(X_n)\cdot v\big)v+X_{n,n} \Big)
  \quad\text{where}\quad
    s(X_n)=\frac{X_{n,1}+\cdots+X_{n,n}}{n}.
    \]
    In other words
    (recall that $\varphi(x) = x\cdot v - c$)
    \[
    \widetilde{X}_n
    =\left(  X_{n,1} - \frac{1}{n}\sum_{i=1}^n\varphi(X_{n,i})v,\ldots,
    X_{n,n}-\frac{1}{n}\sum_{i=1}^n\varphi(X_{n,i})v \right),
    \]
    so that
  \[
    \widetilde\mu_n=\frac{1}{n}\sum_{i=1}^n\delta_{\widetilde X_{n,i}}
    =\mu_n*\delta_{m_n}
    \quad\text{where}\quad
    \mu_n=\frac{1}{n}\sum_{i=1}^n\delta_{X_{n,i}}
    \quad\text{and}\quad
    m_n=v\int\varphi\,\mathrm{d}\mu_n.
   \]
   Thanks to the assumptions on~$V$ and~$g$ we know that the equilibrium
   measure~$\mus$ is the uniform distribution on a ball of radius~$1$. Now we
   note that $\|\varphi\|_{\mathrm{Lip}}\leq 1$ and
   $\int\varphi\,\mathrm{d}\mus = c$, so that by Lemma \ref{le:w} and denoting by
   $\mu^\varphi=\delta_{cv}*\mus$, for all $p\geq1$, 
  \begin{equation}\label{eq:le:wpwp}
    \DWP(\widetilde\mu_n,\mu^\varphi)
    \leq 2^{1-1/p} (1+|v|^p)^{1/p}\DWP(\mu_n,\mus) = 2 \DWP(\mu_n,\mus).
  \end{equation}
  
  On the other hand, the large deviations principle -- see~\cite[Proof of
  Theorem 1.1(4)]{MR3262506} and~\cite{dupuis} for the fact that the
  condition~\eqref{eq:betandiverge} ensures that~$(\beta_n)_n$ diverges fast
  enough -- gives, for any $\varepsilon>0$,
  \begin{equation}\label{eq:le:bl}
    \sum_n\mathbb{P}\big(\DBL(\mu_n,\mus)\geq\varepsilon\big)<\infty.
  \end{equation}
  Alternatively we could use the concentration of
  measure~\cite[Theorem~1.5]{MR3820329} and get the result for~$\DW$ as well.
  This summable convergence in probability towards a non-random limit, known
  as \emph{complete convergence}~\cite{MR1632875}, is equivalent, via
  Borel--Cantelli lemmas, to stating that almost surely,
  $\lim_{n\to\infty}\DBL(\mu_n,\mus)=0$, regardless of the way we defined the
  random variables~$X_n$ and thus the random measures~$\mu_n$ on the same
  probability space.

  In order to upgrade the convergence from~$\DBL$ to~$\DWP$ for all $p\geq1$
  we note that, from~\cite[Theorem 1.12]{MR3820329}, there exists
  $r_0>0$ 
  such that, for all $r\geq r_0$, 
  \begin{equation}\label{eq:le:max}
    \sum_n\mathbb{P}\left(\max_{1\leq k\leq n}|X_{n,k}|\geq r\right)
	  <\infty.
  \end{equation}
  Now for all $p\geq1$ and all $\mu,\nu\in\mathcal{P}(\mathbb{R}^d)$ supported
  in the ball of $\mathbb{R}^d$ of radius $r\geq1$, we have
  \[
    \DWP^p(\mu,\nu)
    \leq (2r)^{p-1}\DW(\mu,\nu)
    \leq r(2r)^{p-1}\DBL(\mu,\nu).  
  \]
  Also, by combining \eqref{eq:le:bl} and
  \eqref{eq:le:max}, we obtain that for all $p\geq1$ and all $\varepsilon>0$,
  \begin{equation}\label{eq:le:wp}
    \sum_n\mathbb{P}\big(\DWP(\mu_n,\mus)\geq\varepsilon\big)<\infty.
  \end{equation}
  By the Borel--Cantelli lemma, for all $p\geq1$, almost surely,
  $\lim_{n\to\infty}\DWP(\mu_n,\mus)=0$, regardless of the way we define the
  random variables~$X_n$ on the same probability space. Finally, since
  $p\mapsto\DWP$ is monotonic in~$p$, we can make the almost sure event valid
  for all~$p$ by taking the intersection of all the almost sure events
  obtained for integer values of~$p$.

  By combining \eqref{eq:le:wp} with \eqref{eq:le:wpwp}, we obtain that for
  all $p\geq1$ and $\varepsilon>0$,
  \begin{equation}\label{eq:le:wp:tilde}
    \sum_n\mathbb{P}\big(\DWP(\widetilde\mu_n,\mu^\varphi)
    \geq\varepsilon\big)<\infty.
  \end{equation}
  Now if~$Y_n$ is a random vector of $(\mathbb{R}^d)^n$ such that
  $\mathrm{Law}(Y_n)=\mathrm{Law}(X_n\mid
  \varphi(X_{n,1})+\cdots+\varphi(X_{n,n})=0)$ then, denoting by
  $\mu_n^Y=\frac{1}{n}\sum_{i=1}^n\delta_{Y_{n,i}}$,
  using~\eqref{eq:le:xtilde} and the fact that $\mu^{\varphi}$ is
  deterministic, we get
  \[
    \DWP(\mu_n^Y,\mu^{\varphi})\overset{\mathrm{d}}
    {=}\DWP(\widetilde\mu_n,\mu^{\varphi}).
  \]
  Therefore, from \eqref{eq:le:wp:tilde}
  we get, for all $p\geq1$ and all $\varepsilon>0$,
  \[
    \sum_n\mathbb{P}\big(\DWP(\mu_n^Y,\mu^\varphi)\geq\varepsilon\big)<\infty.
  \] 
  By the Borel--Cantelli lemma, for all $p\geq1$, almost surely,
  $\lim_{n\to\infty}\DWP(\mu_n^Y,\mu^\varphi)=0$, regardless of the way we
  define the random variables~$Y_n$ on the same probability space. Finally,
  since~$\DWP$ is monotonic in~$p$, we can make the almost sure event valid
  for all~$p$ by taking the intersection of all the almost sure events
  obtained for integer values of~$p$.
\end{proof}

Note that the above proof relies crucially, via Lemma \ref{le:factor}, on the
quadratic nature of~$V$. However the Coulomb nature of~$g$ is less crucial and
the result should remain essentially valid provided that the convergence to
the equilibrium measure holds, for instance at the level of generality of the
assumptions of the large deviations principle in~\cite{MR3262506}.


\section{Proofs of Section~\ref{sec:Gibbs}}
\label{sec:proofLDP}

We start with the proof of the abstract Gibbs conditioning principle.

\begin{proof}[Proof of Theorem~\ref{th:Gibbs}]
  Since $I$ is a good rate function, it is lower semicontinuous with compact level sets.	
  The set~$\mathscr{I}_{B}$ defined in~\eqref{eq:Imin} is not empty because
  the infimum is finite by~\eqref{eq:Icontinuity},~$B$ is closed and~$I$ has
  compact level sets, 
  so the infimum is attained at least for one measure.
  Moreover,~$\mathscr{I}_B$ is closed by lower semicontinuity of~$I$. Now,
  since
  \[
  \frac{1}{\beta_n}\log \dP \left( Z_n \in A_{\varepsilon}\
  \Big| \ Z_n\in B\right) =
  \frac{1}{\beta_n}\log \dP \left( Z_n \in A_{\varepsilon}\cap B\right)
  -   \frac{1}{\beta_n}\log \dP \left( Z_n \in B\right), 
  \]
  the result follows from an upper bound on~$\dP( A_{\varepsilon}\cap B)$ and
  a lower bound on~$\dP(B)$. The upper bound of the large deviations principle
  implies that
  \begin{equation}
    \label{eq:upperinter}
    \limsup_{n\to +\infty}\frac{1}{\beta_n}\log \dP \left( Z_n \in A_{\varepsilon}\cap B\right)
    \leq - \inf_{\overline{A_{\varepsilon}\cap B}} I.
  \end{equation}
  Assume first that~$\overline{A_{\varepsilon}\cap B}\neq \emptyset$. Since
  $A_{\varepsilon}=\{ z\in \mathcal{Z},\ \mathrm{d}(z, \mathscr{I}_{B})>
  \varepsilon \}$, the lower semi-continuity of~$I$ shows that
  (see~\cite[Section~2.5]{MR3262506}) there exists $c_{\varepsilon}>0$ for
  which
  \[
    \inf_{\overline{A_{\varepsilon}\cap B}} I \geq c_{\varepsilon} + \inf_B I,
  \]
  so that
  \begin{equation}
    \label{eq:upperb}
    \limsup_{n\to +\infty}\frac{1}{\beta_n}\log \dP \left( Z_n \in A_{\varepsilon}\cap B\right)
    \leq - \inf_B I - c_{\varepsilon}.
  \end{equation}
  If $\overline{A_{\varepsilon}\cap B}= \emptyset$, the infimum in the right
  hand side of~\eqref{eq:upperinter} is equal to~$+\infty$ so
  that~\eqref{eq:upperb} still holds. The lower bound for the set~$B$ reads
  \[
  \liminf_{n\to +\infty} \frac{1}{\beta_n}\log \dP \left( Z_n \in B\right)
  \geq - \inf_{\mathring{B}} I.
  \]
  Since $B$ satisfies~\eqref{eq:Icontinuity}, it holds
  \[
  \limsup_{n\to +\infty} - \frac{1}{\beta_n}\log \dP \left( Z_n \in B\right)
  \leq \inf_{B} I,
  \]
  which, together with~\eqref{eq:upperb}, leads to~\eqref{eq:GibbsI}.

  Finally, if $Z_n'\sim\mathrm{Law}(Z_n\mid Z_n\in B)$ and if we define
  the~$Z_n'$'s on the same probability space then, for all $\varepsilon>0$,
  thanks to \eqref{eq:betandiverge}, by the Borel--Cantelli lemma,
  $\sum_n\mathbb{P}(Z_n'\in A_\varepsilon)<\infty$ and thus, almost surely,
  $Z_n'\not\in A_\varepsilon$ for large enough~$n$. Since the
  set~$A_\varepsilon$ depends on~$\varepsilon>0$, by taking $\varepsilon\to 0$
  with $\varepsilon\in\dQ$, we obtain that almost surely,
  $\lim_{n\to\infty}\mathrm{d}(Z_n',\mathscr{I}_B)=0$.
\end{proof}

We next recall elements of proof for the properties of the rate function~$\cE$.

\begin{proof}[Proof of Proposition~\ref{prop:cE}]
  Consider a probability measure $\mu\in D_{\cE}$, so
  \[
  \int_{\dR^d}V(x)\mu(\mathrm{d}x) + \iint_{\dR^d\times \dR^d} g(x-y)\mu(\mathrm{d}x)
  \mu(\mathrm{d}y)<+\infty.
  \]
  Since $V$ satisfies Assumption~\ref{as:Vq} and therefore beats~$g$ at
  infinity (in particular when~$d=2$), we have 
  \[
  \int_{\dR^d} |x|^p\mu(\mathrm{d}x) < +\infty,
  \]
  for $1< p<q$. Thus $D_{\cE}\subset \mathcal{P}_p(\dR^d)$.
  
  The strict convexity of $\cE$ is due to a Bochner-type positivity of the
  interaction kernel. See for instance \cite[Chapter~I, Lemma~1.8]{MR1485778},
  \cite{MR1606719}, or \cite[Section~3]{MR3262506} for $d=2$ and
  \cite[Theorems~1.15 and~1.16]{MR0350027} or \cite[Lemma~3.1]{MR3262506} for
  $d\geq3$. The uniqueness of the minimizer follows from the strict convexity.
  The compactness of its support relies on the behaviour of $V$ at infinity,
  see for instance \cite[Chapter~I, Theorem~1.3]{MR1485778} for $d=2$
  and \cite[Theorem~1.2]{MR3262506} for $d\geq3$.
  Finally, since the minimizer of~$\cE$ over $\mathcal{P}(\dR^d)$ has compact
  support, the three problems in~\eqref{eq:minimizers} clearly coincide.
  %
\end{proof}

We finally present the proof of Corollary~\ref{co:Gibbs}, which is a
consequence of Theorem~\ref{th:Gibbs} and the Borel--Cantelli lemma.

\begin{proof}[Proof of Corollary~\ref{co:Gibbs}]
  Under~$P_n$, the empirical measure~$\mu_n$ associated to~$X_n$ satisfies a
  LDP in the $p$-Wasserstein topology with good rate function~$\cE$. Since~$B$
  is assumed to be a closed continuity set for the $p$-Wasserstein topology,
  the set~$\mathscr{E}_B$ defined in~\eqref{eq:Emin} is closed and non-empty
  by Theorem~\ref{th:Gibbs}.

  For simplicity we denote by
  \[
  \mu_n^Y =\frac{1}{n} \sum_{i=1}^n \delta_{Y_{n,i}}
  \]
  the empirical measure associated to~$Y_n$, where
  $Y_n\sim\mathrm{Law}(X_n\,|\,\mu_n\in B)$. For any~$\varepsilon>0$, we
  define the set~$A_{\varepsilon}$ as in Theorem~\ref{th:Gibbs}. Then, there
  exists $c_{\varepsilon}>0$ such that
  \[
  \begin{aligned}
  \sum_n\dP\big( \DWP(\mu_n^Y,\mathscr{E}_B) > \varepsilon \big)
  & = \sum_n\dP\big( \DWP(\mu_n,\mathscr{E}_B) > \varepsilon\ \big| \ \mu_n\in B \big)
  \\ & = \sum_n\dP\big( \mu_n \in A_{\varepsilon}\ \big| \ \mu_n\in B \big)
  \\ & \leq C \sum_n\mathrm{e}^{-\beta_n c_{\varepsilon}} < +\infty,
  \end{aligned}
  \]
  for some~$C>0$. Since~$\beta_n\gg n$ thanks to \eqref{eq:betandiverge}, the
  Borell--Cantelli lemma implies that
  \[
    \lim_{n\to\infty}\DWP\big(\mu_n^Y,\mathscr{E}_B\big)=0,
  \]
  almost surely in any probability space, which concludes the proof.
\end{proof}

\section{Proof of Theorem~\ref{th:lincond}}
\label{sec:prooflincond}

The proof is decomposed into four steps. We first show that under
Assumption~\ref{as:intB}, the set~$B$ is an $I$-continuity set for the
electrostatic energy~$\cE$. We next show that any minimizer of~$\cE$ over~$B$
has a compact support, and hence the minimizer is actually unique. The last
two steps characterize the minimizer through~\eqref{eq:muphi}.

\paragraph{Step 1: $I$-continuity}

Let us first show that $B\subset\mathcal{P}_p(\dR^d)$ is closed for the
$p$-Wasserstein topology by showing that
$B^c=\{\mu\in \mathcal{P}_p(\dR^d)\,|\,\mu(\varphi)>0\}$ is open. Take
$\mu\in B^c$ and~$\nu$ such that $\DWP(\mu,\nu)\leq \varepsilon^{\frac{1}{p}}$
for some $\varepsilon>0$. By definition of the $p$-Wasserstein distance it
holds
\begin{equation}
  \label{eq:wassphi}
\sup_{\substack{f\in\mathrm{L}^1(\mu),\, g\in\mathrm{L}^1(\nu)\\
    f(x)\leq g(y)+|x-y|^p}}\left(\int_{\dR^d} f\,\mathrm{d}\mu
-\int_{\dR^d} g\,\mathrm{d}\nu\right)\leq \varepsilon.
\end{equation}
Since $\|\varphi\|_{\infty,p}<+\infty$, for any
$\mu,\nu\in\mathcal{P}_p(\dR^d)$ it holds $\varphi\in L^1(\nu)\cap L^1(\mu)$.
Moreover, $\|\varphi\|_{\mathrm{Lip}}<+\infty$ and~$\varphi$ cannot be a
constant function because this would contradict the existence of~$\mu_\pm$ in
Assumption~\ref{as:intB}, so $\|\varphi\|_{\mathrm{Lip}}>0$. As a result, for
$|x-y|\geq 1$ we have
\[
|\varphi(x) - \varphi(y)|\leq \|\varphi\|_{\mathrm{Lip}} |x-y|
\leq \|\varphi\|_{\mathrm{Lip}} |x-y|^p. 
\]
Therefore, $\varphi/\|\varphi\|_{\mathrm{Lip}}$ satisfies the inf-convolution
condition in~\eqref{eq:wassphi} and we may pick $f=g=\varphi/\|\varphi\|_{\mathrm{Lip}}$ so that
\[
\int_{\dR^d} \frac{\varphi}{\|\varphi\|_{\mathrm{Lip}}} \,\mathrm{d}\mu
- \int_{\dR^d} \frac{\varphi}{\|\varphi\|_{\mathrm{Lip}}} \,\mathrm{d}\nu
\leq \varepsilon,
\]
which becomes
\[
\nu(\varphi)\geq \mu(\varphi) -\varepsilon \|\varphi\|_{\mathrm{Lip}} >0
\quad \mathrm{for}\quad \varepsilon < \frac{\mu(\varphi)}{\|\varphi\|_{\mathrm{Lip}}}. 
\]
As a result, $B^c$ is open and~$B$ is closed for the $p$-Wasserstein topology.

We now prove that~$B$ is an $I$-continuity set, namely
that~\eqref{eq:Icontlin} holds. By the same reasoning as above, the existence
of $\mu_-\in D_{\cE}$ such that $\mu_-(\varphi)<0$ ensures that
$\mu_-\in \mathring{B}$ (recall that $\mu_- \in \mathcal{P}_p(\dR^d)$ by Proposition~\ref{prop:cE}) so
\[
\inf_{\mathring{B}}\cE < +\infty.
\]
Since~$B$ is closed and~$\cE$ has compact level sets, 
there exists~$\bar{\mu}$ such that
\[
\cE(\bar{\mu}) = \inf_B\cE.
\]
If $\bar{\mu}(\varphi)<0$, it holds $\bar{\mu}\in \mathring{B}$ and the proof
is complete. Thus we may assume that $\bar{\mu}(\varphi)=0$ and, by
considering a minimizing sequence, for any~$\varepsilon>0$ we may find
$\mu_\varepsilon\in\mathring{B}$ such that
\begin{equation}
  \label{eq:approxeps}
\cE(\mu_\varepsilon) \leq \inf_{\mathring{B}}\cE +\varepsilon.
\end{equation}
For $t\in[0,1]$, we introduce
$\mu_t = t\mu_\varepsilon + (1 - t)\bar{\mu}\in D_{\cE}$. Since
$\bar{\mu}(\varphi)=0$ it holds $\mu_t\in \mathring{B}$ for any $t\in (0,1]$.
By convexity of~$\cE$ on its domain we have
\[
\cE(\mu_t)\leq t \cE(\mu_\varepsilon) + (1-t)\cE(\bar{\mu}).
\]
We now proceed by contradiction by assuming that
$\cE(\bar{\mu})= \inf_{\mathring{B}}\cE-\eta$ for some $\eta>0$.
Recalling~\eqref{eq:approxeps} we have, for some $\varepsilon>0$ and any $t\in(0,1]$,
\[
\cE(\mu_t) \leq t \left(\inf_{\mathring{B}}\cE + \varepsilon\right) + (1-t)
    \left(\inf_{\mathring{B}}\cE - \eta\right)
= \inf_{\mathring{B}}\cE + t\varepsilon - (1-t)\eta.
\]
Considering
\[
t < \frac{\eta}{\varepsilon + \eta},
\]
we obtain that $\mu_t\in\mathring{B}$ with
\[
\cE(\mu_t) < \inf_{\mathring{B}}\cE,
\]
which is a contradiction. Therefore,~\eqref{eq:Icontlin} holds true.

\paragraph{Step 2: the minimizer is unique and has compact support}
We now show that any minimizer~$\mu^\varphi$ has a compact support, before
turning to uniqueness. We detail the proof for $d\geq 3$
following~\cite{MR3262506} by highlighting the necessary modifications, and
leave the proof for $d=2$ to the reader (which is deduced
from~\cite[Chapter~I, Theorem~1.3]{MR1485778}). We introduce
\[
\zeta = \inf_{B}\cE
\]
and, for any compact~$K$,
\[
 \zeta_{K}= \inf_{B_K} \cE,\quad
\text{where}
\quad
B_K = \{\mu\in B\ | \ \mathrm{supp}(\mu)\subset K\}.
\]
By Assumption~\ref{as:intB},~$B_K$ is non empty for~$K$ large enough
(consider~$\mu_-(\,\cdot\, \ind_K)/\mu_-(K)$ for~$K$ large enough). By
Assumption~\ref{as:Vq}, for any constant~$C$ the set
\begin{equation}
  \label{eq:defsetK}
K =\left\{x\in\dR^d,\ V(x)  \leq C\right\}
\end{equation}
is compact. In all what follows, we assume that $V \geq 0$. Since~$V$ is lower
bounded and defined up to a constant, there is no loss of generality in this
assumption.

Let us show that $\zeta=\zeta_K$ for~$C$ large enough when~$K$ is defined
by~\eqref{eq:defsetK}. Since the infimums
on~$B$ and~$\mathring{B}$ coincide, we can consider a measure $\mu\in B$ such
that $\mu(\varphi)<0$ and $\cE(\mu)\leq \zeta +1$. If $\mu(K)=1$, the measure
has compact support and we are done, so we assume that $\mu(K) < 1$. The goal
of the following computations is to build a measure~$\mu_K\in B$ supported
in~$K$ such that $\cE(\mu_K)< \cE(\mu)$; this contradiction will show
that~$\zeta$ and~$\zeta_K$ are equal. Let us first show that $\mu(K)>0$
for~$C$ large enough. Indeed,
\[
\zeta + 1 \geq \cE (\mu)= \underbrace{\int_K V \,\mathrm{d}\mu}_{\geq 0}
+ \int_{K^c} V \,\mathrm{d}\mu + \underbrace{J(\mu)}_{\geq 0}
\geq C (1 - \mu(K)),
\]
which shows that $\mu(K)>0$ if $C > \zeta+1$. We may therefore define the restriction
\[
\mu_K(\,\cdot\,) = \frac{ \mu(K\cap \cdot\,)}{\mu(K)}.
\]
Since $\mu(K)<1$, we define similarly $\mu_{K^c}$. The measure~$\mu$ then reads
\[
\mu = \mu(K)\mu_K + (1 - \mu(K))\mu_{K^c}.
\]
Moreover, we chose~$\mu$ such that $\mu(\varphi)<0$, so it holds $\mu_K\in B_K$
for~$C$ large enough. Using the positivity of~$V$ and~$J$ (since $d\geq 3$ and
so $g\geq 0$) and $\mu(K)<1$, we obtain that
\[
\begin{aligned}
  \cE(\mu) = &\ \mu(K)\int_{\dR^d} V\,\mathrm{d}\mu_K + (1-\mu(K))
  \underbrace{\int_{\dR^d} V\,\mathrm{d}\mu_{K^c}}_{\geq C} + \mu(K)^2J(\mu_K) 
  \\ & + 2\mu(K)(1 - \mu(K))\underbrace{J(\mu_K, \mu_{K^c})}_{\geq 0}
  + (1 - \mu(K))^2\underbrace{J( \mu_{K^c})}_{\geq 0} 
  \\  \geq &\ \mu(K)^2J(\mu_K) + \mu(K)^2\int_{\dR^d} V\,\mathrm{d}\mu_K + (1 - \mu(K))C
  \\ \geq &\ \mu(K)^2 \cE(\mu_K) + (1 - \mu(K))C.
  \end{aligned}
  \]
  Let us proceed by contradiction by assuming that $\cE(\mu_K)\geq \cE(\mu)$, which leads to
  \[
  \cE(\mu)\geq \mu(K)^2 \cE(\mu) + (1 - \mu(K))C.
  \]
  Since $\cE(\mu)\leq \zeta +1$ we obtain
  \[
  (\zeta+1)( 1 - \mu(K)^2)\geq  (1 - \mu(K))C.
  \]
  Simplifying by $1 - \mu(K)$ we have
  \[
  2( \zeta + 1) \geq C,
  \]
  which is absurd for $C > 2(\zeta+1)$. Since $\mu_K\in B$ for~$C$ large
  enough, this shows that~$\zeta$ and~$\zeta_K$ coincide and that any
  minimizer has compact support.

  In the above proof, the only modification with respect to previous works
  (see for instance~\cite{MR3262506}) is to check that the restricted
  measure~$\mu_K$ satisfies the constraint for~$C$ large enough. This is done
  by picking the measure~$\mu$ close to the minimum and such
  that~$\mu(\varphi)<0$. The same strategy can be used in the situation where
  $d=2$ by writing
  \[
  \cE(\mu) = \iint_{\dR^d\times \dR^d}
  \left(\frac{V(x) + V(y)}{2} - \log|x-y|\right)\mu(\mathrm{d}x)\mu(\mathrm{d}y),
  \]
  and adapting~\cite[Chapter~I, Theorem~1.3]{MR1485778} since $V(x)+V(y)$
  dominates~$\log|x-y|$ at infinity by Assumption~\ref{as:Vq}.

  The minimizer is unique due to the strict convexity of $\cE$, see Proposition
  \ref{prop:cE}, and the linearity of the constraint.
  
\paragraph{Step 3: Lagrange multiplier}
We now turn to a first step towards the expression of~$\mu^\varphi$ involving
a Lagrange multiplier~$\alpha$. We adapt the proof
of~\cite[Theorem~3.1]{barbu2012convexity} by introducing the following subset
of~$\dR^2$:
\begin{equation}\label{eq:setR}
  R = \big\{ \big(\cE(\mu) - \cE(\mu^\varphi) + a_0, \mu(\varphi) + a_1\big):
    a_0>0, \, a_1 >0,\, \mu\in D_{\cE}\big\}.
\end{equation}
Since~$\cE$ is convex on its domain~$D_{\cE}$ (which is convex),~$R$ is a non
empty convex subset of~$\dR^2$ that does not contain~$(0,0)$ (recall also that
$D_{\cE}\subset \mathcal{P}_p(\dR^d)$ so the constraint takes finite values).
Separating~$R$ from~$(0,0)$ with a hyperplane
(see~\cite[Corollary~1.41]{barbu2012convexity}), this ensures the existence of
$(\alpha_0,\alpha_1)\in\dR^2 \setminus \{(0,0)\}$ such that, for any
$\mu \in D_{\cE}$ and $a_0,a_1>0$ it holds
\[
\alpha_0  \big(\cE(\mu) - \cE(\mu^\varphi) + a_0\big) + \alpha_1\big(\mu(\varphi) + a_1\big) > 0.
\]
By taking $a_1\to-\mu^\varphi(\varphi)\geq 0$ and $\mu=\mu^\varphi$ in the above equation, we obtain
that $\alpha_0 \geq 0$. Then, choosing $\mu=\mu^\varphi$, $a_0 \to 0$ and
$a_1 > -\mu^\varphi(\varphi)\geq 0$ we find $\alpha_1\geq 0$. Taking $a_0,a_1\to 0$, we obtain
\begin{equation}
  \label{eq:varconst0}
    \forall \,\mu\in D_{\cE},\quad \alpha_0 \cE(\mu^\varphi)\leq\alpha_0 \cE(\mu)
    + \alpha_1\mu( \varphi).
\end{equation}
We now prove that $\alpha_0 >0$ by contradiction. If $\alpha_0 = 0$,~\eqref{eq:varconst0}
becomes
\[
\forall \,\mu\in D_{\cE},\quad 0 \leq \alpha_1\mu( \varphi).
\]
Since $\alpha_1\neq 0$ in this case, the above equation contradicts Assumption~\ref{as:intB}
by taking $\mu=\mu_-$,  so $\alpha_0 > 0$ and we may renormalize~\eqref{eq:varconst0} into
\begin{equation}
    \label{eq:varconst}
    \forall \,\mu\in D_{\cE},\quad  \cE(\mu^\varphi)\leq \cE(\mu)
    + \alpha\mu( \varphi),
\end{equation}
where we set $\alpha = \alpha_1/\alpha_0\geq 0$.

Finally, we show that either $\mu^\varphi = \mus$, in which case $\alpha=0$,
or $\mu^\varphi(\varphi)=0$ and $\alpha >0$. First, if $\mus\in B$, $\mus$
satisfies the constraint and we know from Proposition~\ref{prop:cE} that it
solves~\eqref{eq:muphi} with~$\alpha =0$. Otherwise, it holds
$\mus(\varphi)>0$. Since $\cE(\mus) < \cE(\mu^\varphi)$ because~$\mus$ is
the unique global minimizer of~$\cE$, we then obtain from~\eqref{eq:varconst} with
the choice $\mu=\mus$ that~$\alpha>0$. Choosing next $\mu = \mu^\varphi$
in~\eqref{eq:varconst} shows that $\mu^\varphi(\varphi)=0$, so the minimizer
actually saturates the constraint.
  
\paragraph{Step 4: potential equation}
In order to derive the equation for $\mu^\varphi$, we
follow~\cite[Section~4]{MR3262506} by introducing the modified potential and
electrostatic energy, for $\mu\in\mathcal{P}(\dR^d)$,
\[
  V_{\alpha} = V + \alpha \varphi,\quad
  \cE_{\alpha}(\mu) = \int_{\dR^d} V_{\alpha}(x)\mu(\mathrm{d}x) + J(\mu).
\]
Since the case when~$\alpha = 0$ corresponds to no-conditioning and we already
know that the equation is satisfied by the equilibrium measure in this case, we
restrict our attention to the situation in which~$\alpha>0$ and
$\mu^\varphi(\varphi)=0$. We define next, for any $\mu\in D_{\cE}$,
\[
  \forall \,t\in (0,1), \quad \psi(t)= \cE_{\alpha}\big( (1-t)\mu^\varphi + t \mu\big).
\]
Because of~\eqref{eq:varconst} and the convexity of~$\psi$, it holds
$\psi'(0)\geq 0$, so that
  \[
    \begin{aligned}
      0 & \leq\psi'(0)=
      \int_{\dR^d} V_{\alpha}\,\mathrm{d}(\mu - \mu^\varphi)+ 2J(\mu^{\varphi},\mu - \mu^\varphi) 
      \\ & \leq \int_{\dR^d} V_{\alpha}\,\mathrm{d}\mu + 2J(\mu^{\varphi},\mu) -
      \left(\int_{\dR^d} V_{\alpha}\,\mathrm{d} \mu^\varphi+ 2J(\mu^{\varphi})\right) 
      \\ &\leq \int_{\dR^d} (V_{\alpha} + 2 U_{\mu^\varphi})\,\mathrm{d}\mu - C_\varphi,
    \end{aligned}
  \]
  where we set $U_{\mu^\varphi} = \mu^\varphi * g$ and
  \[
    C_\varphi = \int_{\dR^d} V_{\alpha}\,\mathrm{d} \mu^\varphi+ 2J(\mu^{\varphi})
    = \int_{\dR^d} ( V + 2U_{\mu^\varphi})\,\mathrm{d}\mu^\varphi,
  \]
  since $\mu^\varphi(\varphi)=0$.  The above inequality may be rewritten as
  \[
    \forall\,\mu\in D_{\cE},\quad \int_{\dR^d} (V_\alpha + 2U_{\mu^\varphi} - C_\varphi)\,\mathrm{d}\mu
    \geq 0,
  \]
  which proves the second line of~\eqref{eq:muphi}. Recall indeed that, by definition,
  a property holds quasi-everywhere if and only if it holds almost surely for
  all probability measures with finite energy. Moreover, the measures~$\mu$
  and~$\mu^\varphi$ belonging to $D_{\cE}\subset\mathcal{P}_p(\mathbb{R}^d)$, they satisfy
  condition~\eqref{eq:log_int_condition} so that the quantity~$J(\mu,\mu^\varphi)$
  is well-defined following the remark below~\eqref{eq:def_J}.

  Let us now prove the first line in~\eqref{eq:muphi} by contradiction. Assume
  that there is $x\in\mathrm{supp}(\mu^\varphi)$ such that
  $V_\alpha(x) + 2U_{\mu^\varphi}(x) > C_\varphi$. Since~$\mu^\varphi$ has
  compact support,~$U_{\mu^\varphi}$ is lower
  semi-continuous~\cite[page~59]{MR0350027}. Since~$V$ is lower
  semi-continuous and~$\varphi$ is Lipschitz hence
  continuous,~$V_\alpha + 2 U_{\mu^\varphi}$ is lower semi-continuous. There
  exists therefore a neighborhood~$\mathcal{U}$ of~$x$ and $\varepsilon>0$
  such that
  \[
  \forall\,x\in\mathcal{U},\quad V_\alpha(x) + 2 U_{\mu^\varphi}(x) \geq C_\varphi + \varepsilon.
  \]
  Integrating with respect to~$\mu^\varphi$ and using $\mu^\varphi(\varphi)=0$ leads to
  \[
  \begin{aligned}
  C_{\varphi}=  \int_{\dR^d}( V_\alpha + 2U_{\mu^\varphi})\,\mathrm{d}\mu^\varphi
  & = \int_{\mathcal{U}}( V + 2U_{\mu^\varphi})\,\mathrm{d}\mu^\varphi
  + \int_{\dR^d\setminus\mathcal{U}}( V + 2 U_{\mu^\varphi})\,\mathrm{d}\mu^\varphi
  \\ & \geq (C_\varphi + \varepsilon)\mu^\varphi(\mathcal{U}) 
  + \int_{\dR^d\setminus\mathcal{U}}( V + 2 U_{\mu^\varphi})\,\mathrm{d}\mu^\varphi.
  \end{aligned}
  \]
  Since $V_\alpha + 2 U_{\mu^\varphi}\geq C_\varphi$ quasi-everywhere and $\mu^\varphi\in D_{\cE}$,
  the above inequality becomes
  \[
  C_\varphi  = \int_{\dR^d}(V_\alpha + 2 U_{\mu^\varphi})\,\mathrm{d}\mu^{\varphi}
  \geq C_\varphi + \varepsilon \mu^\varphi(\mathcal{U}).
  \]
  We reach a contradiction by noting that $\mu^\varphi(\mathcal{U}) >0$ since~$\mathcal{U}$
  is a neighborhood of $x\in\mathrm{supp}(\mu^\varphi)$ (using the definition of
  the support), which proves the first line of~\eqref{eq:muphi}.

\section{Proof of Theorem~\ref{th:quadcond}}
\label{sec:proofquadstats}

We outline the proof of Theorem~\ref{th:quadcond}, which follows the same
lines as in the linear case.
\begin{proof}
  We show below that the set~$B\subset \mathcal{P}_p(\dR^d)$ defined in~\eqref{eq:Bquad} is closed for the
  $p$-Wasserstein topology under Assumption~\ref{as:psi}. For this, we show
  that~$B^c$ is open by picking $\mu\in \mathcal{P}_p(\dR^d)$ such that
  $Q(\mu)>0$ and using again that, for~$\varepsilon>0$ and
  $\nu\in \mathcal{P}_p(\dR^d)$ such that
  $\DWP(\mu,\nu)\leq \varepsilon^{\frac{1}{p}}$, it holds,
  by~\eqref{eq:kantorovitch},
  \begin{equation}
    \sup_{\substack{f\in\mathrm{L}^1(\mu),\, g\in\mathrm{L}^1(\nu)\\
        f(x)\leq g(y)+|x-y|^p}}\left(\int_{\dR^d} f\,\mathrm{d}\mu
      -\int_{\dR^d} g\,\mathrm{d}\nu\right)\leq \varepsilon.
  \end{equation}
  First, by Assumption~\ref{as:psi}, we note that for
  any~$\mu,\nu\in\mathcal{P}_p(\dR^d)$ it
  holds~$\|U_\mu^\psi\|_{\infty,p}<+\infty$ and
  $\|U_\nu^\psi\|_{\infty,p}<+\infty$. Therefore,
  $U_\mu^\psi\in L^1(\mu)\cap L^1(\nu)$ and
  $U_\nu^\psi\in L^1(\mu)\cap L^1(\nu)$ for any probability measures~$\mu,\nu$
  with moments of order~$p$. Next, by~\eqref{eq:psicond}, it holds
  $\|U_\mu^\psi\|_{\mathrm{Lip}}\leq C_{\mathrm{Lip}}$ and
  $\|U_\nu^\psi\|_{\mathrm{Lip}}\leq C_{\mathrm{Lip}}$. We assume for now that
  these norms are non-zero, so we may first choose
  $f=g=U_\mu^\psi/\|U_\mu^\psi\|_{\mathrm{Lip}}$, which leads to
  \begin{equation}
    \label{eq:munu1}
    Q (\mu) -\iint_{\dR^d\times\dR^d}\psi(x,y)\nu( \mathrm{d}x)\mu(\mathrm{d}y)
    \leq \varepsilon\|U_\mu^\psi\|_{\mathrm{Lip}}.
  \end{equation}
  Symmetrically we take $f=g=-U_\nu^\psi/\|U_\nu^\psi\|_{\mathrm{Lip}}$, which
  leads to
  \begin{equation}
    \label{eq:munu2}
    \iint_{\dR^d\times\dR^d}\psi(y,x)\nu( \mathrm{d}x)\mu(\mathrm{d}y)
    - Q (\nu) \leq \varepsilon\|U_\nu^\psi\|_{\mathrm{Lip}}.
  \end{equation}
  By summing~\eqref{eq:munu1} and~\eqref{eq:munu2} and using the symmetry
  of~$\psi$, we obtain
  \begin{equation}
    \label{eq:ineqQQ}
    Q(\nu) \geq Q(\mu) - \varepsilon \big( \|U_\nu^\psi\|_{\mathrm{Lip}}
    + \|U_\mu^\psi\|_{\mathrm{Lip}}\big)\geq Q(\mu) - 2 C_{\mathrm{Lip}} \varepsilon. 
  \end{equation}
  This shows that $Q(\nu)>0$ for $\varepsilon< Q(\mu)/(2C_{\mathrm{Lip}})$. To
  finish the argument, we consider the cases where the Lipschitz norm of the
  potentials generated by~$\mu$ and~$\nu$ may be zero. Suppose first that
  $\|U_\mu^\psi\|_{\mathrm{Lip}}=0$. This implies the existence of
  $c_\mu\in\dR$ such that
  \begin{equation}
    \label{eq:crossterm}
    \forall\, x\in\dR^d,\quad \int_{\dR^d}\psi(x,y)\mu(\mathrm{d}y) = c_\mu.
  \end{equation}
  Integrating the above equation with respect to~$\mu$ shows that
  $c_\mu = Q(\mu)>0$. As a result, if~$\|U_\mu^\psi\|_{\mathrm{Lip}}=0$
  and~$\nu$ is such that~$\|U_\nu^\psi\|_{\mathrm{Lip}}>0$ we can
  consider~\eqref{eq:munu2}, which becomes (integrating~\eqref{eq:crossterm}
  with respect to~$\nu$)
  \[
    Q(\nu) \geq Q(\mu) - \varepsilon\|U_\nu^\psi\|_{\mathrm{Lip}}\geq Q(\mu) - \varepsilon
    C_{\mathrm{Lip}}.
  \]
  In this case, $Q(\nu)>0$ for $\varepsilon < Q(\mu)/C_{\mathrm{Lip}}$. Then,
  if~$\|U_\nu^\psi\|_{\mathrm{Lip}}=0$ it holds, for some~$c_\nu\in\dR$,
  \begin{equation}
    \label{eq:cnunul}
    \forall\, x\in\dR^d,\quad \int_{\dR^d}\psi(x,y)\nu(\mathrm{d}y) = c_\nu.
  \end{equation}
  Integrating with respect to~$\mu$ and using the symmetry of~$\psi$ we obtain
  that $c_\nu = c_\mu>0$. Integrating next~\eqref{eq:cnunul} with respect
  to~$\nu$ shows that~$Q(\nu)= c_\mu=Q(\mu)>0$. Finally, if
  $\|U_\mu^\psi\|_{\mathrm{Lip}}>0$ but
  $\|U_\nu^\psi\|_{\mathrm{Lip}}=0$,~\eqref{eq:cnunul} holds with
  $c_\nu = Q(\nu)$ so that~\eqref{eq:munu1} becomes
  \[
    Q(\nu) \geq Q(\mu) - \varepsilon C_{\mathrm{Lip}},
  \]
  and the same conclusion follows. As a result, in any case the measures~$\nu$
  such that $\DWP(\mu,\nu)\leq \varepsilon^{\frac{1}{p}}$ for
  $\varepsilon < Q(\mu)/(2C_{\mathrm{Lip}})$ belong to~$B^c$ so that~$B^c$ is
  open and~$B$ is closed in the $p$-Wasserstein topology.
  
  We next show that~$B$ is an $I$-continuity set. The existence
  of~$\mu_-\in D_{\cE} \subset \mathcal{P}_p(\dR^d)$ such that $Q(\mu_-)<0$ ensures that
  \[
    \inf_{\mathring{B}} \cE < + \infty.
  \]
  Since~$\cE$ has compact level sets and~$B$ is closed, there
  exists~$\bar{\mu}$ such that
  \[
    \cE(\bar{\mu}) = \inf_B\cE.
  \]
  If $Q(\bar{\mu})<0$, $I$-continuity is proven, so we may assume that
  $Q(\bar{\mu})=0$. Like in the linear case, we may
  take~$\mu_{\varepsilon}\in\mathring{B}$ such that
  \[
    \cE(\mu_\varepsilon)\leq \inf_{\mathring{B}} \cE + \varepsilon,
  \]
  and consider the convex combination
  $\mu_t = t \mu_\varepsilon + (1-t)\bar{\mu}$ for $t\in(0,1)$. The convexity
  of~$Q$ shows that, for any $t\in(0,1)$ it holds
  \[
    Q(\mu_t) \leq t Q(\mu_\varepsilon) + (1-t)Q(\bar{\mu}) <0,
  \]
  so that $\mu_t \in \mathring{B}$. Proceeding by contradiction by supposing
  that $\cE(\bar{\mu})< \inf_{\mathring{B}}\cE$, we obtain that for $t>0$
  small enough it holds $\mu_t \in \mathring{B}$ and
  \[
    \cE(\mu_t) < \inf_{\mathring{B}} \cE,
  \]
  which is a contradiction, proving that~$B$ is an $I$-continuity set.

  One can next follow Step~2 of the proof of Theorem~\ref{th:lincond} to show
  that the minimizer~$\mu^\psi$ is unique with compact support.
  
  At this stage, the remaining statements in Theorem~\ref{th:quadcond} can be
  proved as for Theorem~\ref{th:lincond}. In particular, we can introduce a
  set similar to~\eqref{eq:setR} by setting
  \[
    R = \big\{ \big(\cE(\mu) - \cE(\mu^\varphi) + a_0, Q(\mu) + a_1\big):
    a_0>0, \, a_1 >0,\, \mu\in D_{\cE}\big\}.
  \]
  The set~$R$ is convex by convexity of~$Q$, so the same convex separation
  theorem can be used, and we can show that there exists~$\alpha \geq 0$ such
  that
  \[
    \forall\, \mu\in D_{\cE}, \quad \cE(\mu^\psi) \leq \cE(\mu) +\alpha Q(\mu).
  \]
  In this procedure, we use the existence of~$\mu_\pm$ from
  Assumption~\ref{as:psi} in order to reproduce the qualification of
  constraint argument. This leads to computations where the interaction
  energy~$J$ is replaced by
  \[
    J_{\alpha}(\mu,\nu)=J(\mu,\nu) + \alpha \iint_{\dR^d\times\dR^d}
    \psi(x,y)\mu(\mathrm{d}x)\mu(\mathrm{d}y),
  \]
  from which~\eqref{eq:mupsi} follows by mimicking Step~4 of the proof of
  Theorem~\ref{th:lincond}.
\end{proof}

As a final comment, let us insist on the importance of the convexity of~$Q$
for the above proof to be valid.


\bibliographystyle{siamplain}


\providecommand{\bysame}{\leavevmode ---\ }
\providecommand{\og}{``}
\providecommand{\fg}{''}
\providecommand{\smfandname}{\&}
\providecommand{\smfedsname}{eds.}
\providecommand{\smfedname}{ed.}
\providecommand{\smfmastersthesisname}{Memoir}
\providecommand{\smfphdthesisname}{Thesis}

\end{document}